\documentclass[12pt,reqno,a4paper]{amsart}
\usepackage{amsmath,amssymb,amsthm}
\usepackage[OT1]{fontenc}

\newcommand{\review}[1]{#1}
\newenvironment{reviewchange}{\begingroup}{\endgroup}
\usepackage{xcolor}
\newif\ifroundreview
\roundreviewtrue

\usepackage[textwidth=18.2cm,textheight=22.2cm,centering]{geometry}
\usepackage{hyperref}
\hypersetup{colorlinks=true,linkcolor=blue,citecolor=blue,urlcolor=blue}
\newcommand{\mathds}[1]{\mathbf{#1}}

\makeatletter
\renewcommand\section{\@startsection{section}{1}{\z@}%
  {2.5ex plus 1ex minus .2ex}%
  {1.5ex plus .2ex}%
  {\centering\normalfont\large\bfseries}}
\renewcommand\subsection{\@startsection{subsection}{2}{\z@}%
  {2.0ex plus .8ex minus .2ex}%
  {1.0ex plus .2ex}%
  {\normalfont\normalsize\bfseries}}
\makeatother

\theoremstyle{plain}
\newtheorem{theorem}{Theorem}
\newtheorem{lemma}{Lemma}
\newtheorem{corollary}{Corollary}
\newtheorem{proposition}{Proposition}

\theoremstyle{definition}

\theoremstyle{remark}
\newtheorem{remark}{Remark}
\theoremstyle{definition}
\newtheorem*{thmA}{\textbf{Theorem A}}
\theoremstyle{definition}

\DeclareMathOperator{\PV}{P.V.}

\DeclareMathOperator{\Ind}{Ind}
\DeclareMathOperator{\Area}{Area}

\newcommand{\C}{\mathbb{C}}
\newcommand{\R}{\mathbb{R}}
\newcommand{\HH}{\mathbb{H}}
\newcommand{\Ch}{\widehat{\mathbb{C}}}
\newcommand{\eps}{\varepsilon}
\newcommand{\dd}{\mathrm{d}}
\newcommand{\dA}{\,dA}
\numberwithin{equation}{section}
\title[Quasiconformal deformations of chordal Loewner chains]%
{Quasiconformal deformations of chordal Loewner chains}

\author{Yong Han}
\address{School of Mathematical Sciences, Shenzhen University,\newline
Shenzhen 518060, Guangdong, China}
\email{hanyong@szu.edu.cn}

\author{Renxin Jiang}
\address{School of Mathematical Sciences, Chongqing University,\newline
Chongqing 400044, China}
\email{RenxinJiang321@outlook.com}
\date{October 1, 2026}

\begin{document}

\begin{abstract}
We establish first-order variational formulas for the driving function
and the half-plane capacity of a chordal Loewner chain under
quasiconformal deformations.  Under the standard normalization fixing
$0,1$ and $\infty$, the formulas hold for every bounded Beltrami
differential.  If the curve has positive area, the formulas contain
additional integrals over the curve; these terms vanish when the curve
has zero area or when the differential is extended by zero on the curve.  
The proof uses the normalized Ahlfors--Bers variation together with
a Stoilow factorization adapted to the slit domain.

We then consider the asymptotic normalization for differentials
whose symmetric extensions to $\mathbb C$ have the form
\[
\nu(z)=a\Bigl(\frac{\bar z}{z}\Bigr)^{n-1}+\nu_0(z),
\qquad
a\in(-1,1),\quad n\in\mathbb N,\quad n\ge1,\quad
\nu_0\in L^r(\mathbb C),\quad 0<r<2.
\]
The asymptotically normalized map differs from the standardly
normalized map by a positive dilation.  Using this relation, the
variational formulas in the second normalization follow directly from
the standardly normalized case.
\end{abstract}

\maketitle

\section{Introduction}

In a recent work \cite{SW24}, Sung and Wang derived explicit
variational formulas for the chordal Loewner driving function and
capacity under quasiconformal deformations.  This is essential for computing the first variation of the Loewner
energy.

In this paper we extend Theorem~1.1 of \cite{SW24} to arbitrary
bounded Beltrami differentials under the standard normalization
fixing $0,1,\infty$, without any compact support or disjointness
assumption. Using singular integral theory, we derive explicit
variational formulas in this setting. We also consider non-decaying
model differentials of the form
$\nu(z)=a(\bar z/z)^{n-1}$, where $a\in(-1,1)$ and $n\ge1$ is an
integer. For $a\ne0$, these differentials belong to no
$L^r(\HH)$ with $0<r<\infty$; they are constant when $n=1$ and
have angular oscillation when $n>1$. Under the asymptotic
normalization specified in Theorem~\ref{thm-main-hh}, we establish
the existence of the first variations of the driving function
and capacity and obtain the corresponding variational formulas.


To be precise, let
$\gamma:[0,T_+)\to\overline{\HH}$ be a simple curve,
parameterized by half-plane capacity, where $0<T_+\le\infty$,
with $\gamma(0)=0$, $\gamma((0,T_+))\subset\HH$, and
$\gamma(t)\to\infty$ as $t\uparrow T_+$.
All assertions at time $t$ below concern $0\le t<T_+$;
no assumption that the total half-plane capacity is infinite is needed.
Let $g_t\colon\HH\setminus\gamma[0,t]\to\HH$ be its Loewner map satisfying
the hydrodynamic normalization
$g_t(z)=z+2t/z+O(1/z^2)$ as $z\to\infty$, with driving function
$\lambda_t=g_t(\gamma(t))$.  Let $\nu\in L^\infty(\HH)$ be a Beltrami
differential.  In the compact-support setting of {\bf Theorem~A} below, for
each $\eps\in\R$ with
$\|\eps\nu\|_\infty<1$, let $\psi^{\eps\nu}$ denote the unique
quasiconformal self-map of $\HH$ with Beltrami coefficient $\eps\nu$,
satisfying $\psi^{\eps\nu}(0)=0$ and the hydrodynamic normalization
\begin{align}\label{eqn-cond0-s}
\psi^{\eps\nu}(z)-z=O(1)\qquad\text{as }z\to\infty .
\end{align}

\begin{thmA}[{see \cite[Theorem~1.1]{SW24}}]\label{thm-wang-sung}
If the support of $\nu$ is compact and disjoint from $\gamma$, then the
following variational formulas hold at $\eps=0$:
\begin{align}
\left.\frac{\partial\lambda_t^{\eps\nu}}{\partial\eps}
\right|_{\eps=0}
&=-\frac{2}{\pi}\operatorname{Re}\iint_{\HH}
\nu(z)\left(
\frac{g_t'(z)^2}{g_t(z)-\lambda_t}-\frac1z
\right)\dA (z),
\label{eqn-Wang-one}\\
\left.\frac{\partial a_t^{\eps\nu}}{\partial\eps}
\right|_{\eps=0}
&=\frac1\pi\operatorname{Re}\iint_{\HH}
\nu(z)\bigl(g_t'(z)^2-1\bigr)\dA (z).
\label{eqn-Wang-two}
\end{align}
Here $\lambda_t^{\eps\nu}$ and $2a_t^{\eps\nu}$ denote the driving
function and the half-plane capacity of the image curve
$\psi^{\eps\nu}(\gamma)$, respectively.
\end{thmA}

As noted in \cite{SW24}, it is a natural and important question whether
the restrictive assumptions on the Beltrami differential---namely,
compact support and
$\operatorname{Supp}(\nu)\cap\gamma=\emptyset$---can be relaxed.  

Notice that for a general Beltrami coefficient,  the $O(1)$ normalization
\eqref{eqn-cond0-s} need not exist: already for a real
constant $\nu=a\neq0$, for $\eps\neq0$ with $|\eps a|<1$,
every quasiconformal self-map of $\HH$ with
coefficient $\eps a$ fixing $0$ and $\infty$ is a positive dilation of
$z\mapsto z+\eps a\bar z$, and the asymptotic identity required by
\eqref{eqn-cond0-s} cannot hold.  For general bounded coefficients we
therefore use two different normalizations:  one is the standard normalization \eqref{eqn-normal-new-han}
below, and the other is the oscillatory asymptotic normalization whose existence is proved in
Theorem~\ref{thm-yibande-n-new}.
To be precise, we show that the compact-support and disjointness
restrictions 
can be entirely removed if one of the following conditions holds:
\begin{itemize}
\item[\textbf{(a)}]
The hydrodynamic normalization \eqref{eqn-cond0-s} is replaced by the following 
standard normalization (Theorem~\ref{thm-main-normal}):
\begin{align}\label{eqn-normal-new-han}
\psi_{\mathrm{std}}^{\eps\nu}(0)=0,\qquad
\psi_{\mathrm{std}}^{\eps\nu}(1)=1,\qquad
\psi_{\mathrm{std}}^{\eps\nu}(\infty)=\infty .
\end{align}
\item[\textbf{(b)}]
The Beltrami differential $\nu$ has the oscillatory asymptotic form
specified in Theorem~\ref{thm-main-hh}.
\end{itemize}

We first prove the result under the standard normalization
\eqref{eqn-normal-new-han}.  We then establish the existence of the
asymptotically normalized solution in Case~\textbf{(b)}.
Finally, the two normalized solutions are related by a dilation, which
allows us to deduce the oscillatory formulas from the standardly
normalized result.

\subsection{Motivation: an explicit example}\label{sec-motiv-exam}
Before developing the general framework, we first test whether the
Sung--Wang variational formulas retain their form in an explicitly
computable non-decaying example. Consider the straight-line segment (see \cite[Example~4.12]{La05}):
\begin{align}\label{eqn-QC-deform-2}
\eta(t)=2\sqrt t\,
\Bigl(\frac{\alpha}{1-\alpha}\Bigr)^{\alpha-1/2}
e^{i\pi(1-\alpha)},\qquad t\ge 0,
\end{align}
where $\alpha\in(0,1)$.  For $t>0$, the normalized conformal mapping
$g_t\colon\HH\setminus\eta[0,t]\to\HH$ is given by
\begin{align*}
g_t^{-1}(z)=
\sqrt{\frac{4t}{\alpha(1-\alpha)}}\;
\phi\!\left(z\sqrt{\frac{\alpha(1-\alpha)}{4t}}\right),
\end{align*}
where
\begin{align*}
\phi(z)=\bigl(z+(1-\alpha)\bigr)^\alpha
        (z-\alpha)^{1-\alpha}.
\end{align*}
Both powers are defined using the principal logarithm, with
$\arg(z+1-\alpha),\arg(z-\alpha)\in(0,\pi)$ for $z\in\HH$.
Boundary values are taken from the upper half-plane.
The map $\phi$ maps $\HH$ onto $\HH$ minus a segment from $0$ to
$w_\alpha=\alpha^\alpha(1-\alpha)^{1-\alpha}
e^{i(1-\alpha)\pi}$.
At $t=0$ we set $g_0=\mathrm{id}$. The initial driving value and
capacity, before and after each deformation below, are zero; hence their
$\varepsilon$-variations at $t=0$ are also zero.

Consider the Beltrami differential
$\nu(z)=a\bar z^{\,n-1}/z^{n-1}$ with $a\in(-1,1)$ and an integer
$n\ge1$.  For sufficiently small $|\eps|$, the asymptotically
normalized quasiconformal mapping on $\HH$ induced by $\eps\nu$ is
explicitly\footnote{Here and below the $n$-th root is understood through
\begin{align*}
\bigl(z^n+\eps a\bar z^{\,n}\bigr)^{1/n}
=
z\Bigl(1+\eps a\Bigl(\frac{\bar z}{z}\Bigr)^{n}\Bigr)^{1/n},
\qquad z\neq0,
\end{align*}
\review{where the second factor uses the principal $n$-th root on the
right half-plane, which contains the closed disk
$\overline{D}(1,|\eps a|)$ because $|\eps a|<1$.
This includes $a=0$, and the map extends continuously by $0\mapsto0$.}}
\begin{align*}
\phi^{\eps\nu}(z)=\bigl(z^n+\eps a\bar z^{\,n}\bigr)^{1/n}.
\end{align*}

Under this mapping, the image curve
$\eta^{\eps\nu}(t)=\phi^{\eps\nu}(\eta(t))$ remains a straight line
segment with a modified length scale and inclination angle:
\begin{align*}
\eta^{\varepsilon\nu}(t)= 2\sqrt{t} \left( \frac{\alpha}{1-\alpha} \right)^{\alpha - 1/2} \left[ 1 + a^2\varepsilon^2 + 2a\varepsilon\cos(2n\pi\alpha) \right]^{\frac{1}{2n}} e^{i\pi(1-\tilde{\theta})},
\end{align*}
Here the inclination parameter has the following expression with a fixed logarithm
\[
\tilde{\theta}
=\alpha-\frac{1}{n\pi}\operatorname{Im}
\log\!\left(1+\varepsilon a e^{2in\pi\alpha}\right),
\]
where the logarithm is principal on the right half-plane. Indeed,
for $t>0$, $\arg\eta(t)=\pi(1-\alpha)$ and
$(\overline{\eta(t)}/\eta(t))^n=e^{2in\pi\alpha}$, so this formula
follows directly from the chosen branch of the model map. It gives
$\tilde{\theta}=\alpha$ at $\varepsilon=0$ and remains valid at all
angles, including the poles of the tangent in the equivalent
arctangent representation.

Using the explicit geometry of $\eta^{\varepsilon\nu}(t)$, its driving function $\lambda^{\varepsilon\nu}(t)$ and half-plane capacity $2a_t^{\varepsilon \nu}$ can be computed directly. Taking their derivatives with respect to $\varepsilon$ at $\varepsilon=0$ yields the geometric variations:
\begin{align}
\left. \frac{d\lambda^{\varepsilon\nu}(t) }{d\varepsilon} \right|_{\varepsilon=0} &= \frac{2a\sqrt{t}}{n\sqrt{\alpha(1-\alpha)}} \left[ (2\alpha-1)\cos(2n\pi\alpha) - \frac{\sin(2n\pi\alpha)}{\pi} \left( 2 - (2\alpha-1)\ln\frac{\alpha}{1-\alpha} \right) \right], \label{eqn-QC-deform-3} \\
\left.\frac{d a_t^{\varepsilon \nu}}{d \varepsilon}\right|_{\varepsilon=0} &= \frac{t a}{n}\left[2 \cos (2 n \pi \alpha)+\frac{\sin (2 n \pi \alpha)}{\pi}\left(\frac{2 \alpha-1}{\alpha(1-\alpha)}+2 \ln \frac{\alpha}{1-\alpha}\right)\right]. \label{eqn-QC-deform-cap-var}
\end{align}
Let $\mathcal L(\nu)$ and $\mathcal S(\nu)$ denote the integrals on the
right-hand sides of \eqref{eqn-Wang-one} and \eqref{eqn-Wang-two},
respectively, interpreted as circular principal values.  The following
two lemmas verify that these integral representations agree with the
explicit derivatives computed above; their proofs are deferred to the
Appendix \ref{sec-appendix-contours}.

\begin{lemma}\label{lemma-QC-deform-1}
For $\alpha\in(1/2,1)$, let
$\nu(z)=a\bar z^{\,n-1}/z^{n-1}$ with $a\in(-1,1)$ and an integer
$n\ge1$.  Then
\begin{align}\label{eqn-QC-deform-4}
\begin{split}
\mathcal L(\nu)
:=
-\frac2\pi\PV\operatorname{Re}
\iint_{\HH}
\nu(z)\left(
\frac{g_t'(z)^2}{g_t(z)-\lambda_t}-\frac1z
\right)\dA (z)
=\text{the right-hand side of \eqref{eqn-QC-deform-3}}
\end{split}
\end{align}
\end{lemma}

\begin{lemma}\label{lemma-QC-deform-2}
Under the same conditions as Lemma~\ref{lemma-QC-deform-1},
\begin{align}
\mathcal S(\nu)
:=
\frac1\pi\PV\operatorname{Re}
\iint_{\HH}
\nu(z)\bigl(g_t'(z)^2-1\bigr)\dA (z)
=\text{the right-hand side of \eqref{eqn-QC-deform-cap-var}}
\end{align}
\end{lemma}

Lemmas~\ref{lemma-QC-deform-1} and \ref{lemma-QC-deform-2} immediately
imply the following corollary.

\begin{corollary}\label{coro-QC-deform}
For the segment $\eta$ in \eqref{eqn-QC-deform-2}
with $\alpha\in(1/2,1)$, and for
$\nu(z)=a\bar z^{\,n-1}/z^{n-1}$ with $a\in(-1,1)$ and an integer
$n\ge1$, the oscillatory first variations have the Sung--Wang form
\eqref{eqn-Wang-one}--\eqref{eqn-Wang-two}, interpreted as circular
principal values.
\end{corollary}

\begin{remark}\label{remark-QC-deform}
Corollary~\ref{coro-QC-deform} does not extend to an arbitrary
non-decaying Beltrami differential.  Consider
$\nu(z)=a z/\bar z$ with $a\in(0,1)$.  The induced quasiconformal
mapping is the radial stretching
$\phi^{\eps\nu}(z)=z|z|^{\frac{2a\eps}{1-a\eps}}$.
The variation of the driving function for the image curve
$\phi^{\eps\nu}(\eta)$ satisfies the following formula for $t>0$;
at $t=0$ its right-hand side is interpreted by continuous extension as
zero, since $\sqrt{t}\log t\to0$:
\begin{align}\label{eqn-radial-stretch-variation}
\left.\frac{\dd\lambda^{\eps\nu}(t)}{\dd\eps}\right|_{\eps=0}
=
\frac{4a(2\alpha-1)\sqrt t}{\sqrt{\alpha(1-\alpha)}}
\left[
\ln2+\frac12\ln t
+\Bigl(\alpha-\frac12\Bigr)\ln\Bigl(\frac{\alpha}{1-\alpha}\Bigr)
\right].
\end{align}
\begin{reviewchange}
For $t>0$ and $\alpha\in(1/2,1)$, the proposed circular
principal value does not exist. Indeed,
\[
\nu(z)\left(\frac{g_t'(z)^2}{g_t(z)-\lambda_t}-\frac1z\right)
=\frac{a\lambda_t}{|z|^2}+O(|z|^{-3}),
\qquad z\to\infty,
\]
so its truncated variation integral equals
\[
-\frac2\pi\operatorname{Re}\iint_{\HH\cap\{|z|<R\}}
\nu(z)\left(\frac{g_t'(z)^2}{g_t(z)-\lambda_t}-\frac1z\right)dA(z)
=-2a\lambda_t\log R+O(1).
\]
Here the integrand is taken off the zero-area slit. The radial map
fixes $0,1,\infty$, so its finite geometric variation is covered by
Theorem~\ref{thm-main-normal}; the normalization term in that theorem
cancels precisely this nonintegrable tail.
\end{reviewchange}
\end{remark}

Corollary~\ref{coro-QC-deform} and Remark~\ref{remark-QC-deform}
motivate the oscillatory class treated in Theorem~\ref{thm-main-hh} below.

\subsection{Main results}\label{sec-main-sdh}

We begin with the standard normalization.  Recall that a canonical,
and perhaps the most commonly used, normalization for solutions of the
Beltrami equation on the Riemann sphere is obtained by fixing the three
marked points $0,1,\infty$.  One advantage of this normalization is that
it is available for an arbitrary bounded Beltrami differential, without
any support or decay assumption at infinity.  It is therefore natural
to establish our general variational formulas first in this setting.

More precisely, let $\nu\in L^\infty(\HH)$ and extend it to
$\C$ by reflection across the real axis:
\begin{align}\label{eqn-symmetric-extension}
\widehat\nu(z)
=
\begin{cases}
\nu(z), & z\in\HH,\\[2mm]
\overline{\nu(\bar z)}, & z\in-\HH.
\end{cases}
\end{align}
For every real $\eps$ satisfying
\[
|\eps|\,\|\nu\|_\infty<1,
\]
the measurable Riemann mapping theorem (see \cite{Ah66,LV73,AIM09}) gives a unique quasiconformal
homeomorphism
\[
\psi_{\mathrm{std}}^{\eps\nu}\colon
\widehat{\C}\longrightarrow\widehat{\C}
\]
whose Beltrami coefficient is $\eps\widehat\nu$ and which satisfies
\begin{align}\label{eqn-standard-normalization-main}
\psi_{\mathrm{std}}^{\eps\nu}(0)=0,\qquad
\psi_{\mathrm{std}}^{\eps\nu}(1)=1,\qquad
\psi_{\mathrm{std}}^{\eps\nu}(\infty)=\infty.
\end{align}
By uniqueness and the symmetry of \eqref{eqn-symmetric-extension},
$\psi_{\mathrm{std}}^{\eps\nu}$ commutes with reflection in $\R$.
Consequently it preserves $\R$ and $\HH$, and hence its restriction
to $\HH$ is the unique quasiconformal self-map of $\HH$ with Beltrami
coefficient $\eps\nu$ satisfying
\eqref{eqn-standard-normalization-main}.

We now introduce the corresponding Loewner quantities.  For a fixed
$\review{0\le t<T_+}$, put
\[
K_t=\gamma[0,t],\qquad 
K_t^\circ=K_t\cap\HH,\qquad
D_t=\HH\setminus\gamma[0,t],
\qquad
K_{t,\mathrm{std}}^{\eps\nu}
=
\psi_{\mathrm{std}}^{\eps\nu}(K_t).
\]
Let
\[
g_{t,\mathrm{std}}^{\eps\nu}\colon
\HH\setminus K_{t,\mathrm{std}}^{\eps\nu}\longrightarrow\HH
\]
\review{be the hydrodynamically normalized conformal map. The index $t$
always denotes the original capacity time: the deformed curve has
half-plane capacity $2a_{t,\mathrm{std}}^{\eps\nu}$ at that same point
of the original parametrization. Thus}
\begin{align}\label{eqn-deformed-hydrodynamic-normalization}
g_{t,\mathrm{std}}^{\eps\nu}(z)
=
z+\frac{2a_{t,\mathrm{std}}^{\eps\nu}}{z}
+O\!\left(\frac1{z^2}\right),
\qquad z\to\infty.
\end{align}
Equivalently,
\[
2a_{t,\mathrm{std}}^{\eps\nu}
=
\operatorname{hcap}
\bigl(K_{t,\mathrm{std}}^{\eps\nu}\bigr).
\]
The corresponding driving point is defined by
\begin{align}\label{eqn-deformed-driving-standard}
\lambda_{t,\mathrm{std}}^{\eps\nu}
=
g_{t,\mathrm{std}}^{\eps\nu}
\bigl(
\psi_{\mathrm{std}}^{\eps\nu}(\gamma(t))
\bigr),
\end{align}
Here the value at the tip is understood through the
prime-end extension of the mapping-out function,
at the prime end corresponding to the tip;
see \cite{Po92}.

At $\eps=0$ we recover the original Loewner chain:
\[
\psi_{\mathrm{std}}^{0}=\operatorname{id},
\qquad
g_{t,\mathrm{std}}^{0}=g_t,
\qquad
a_{t,\mathrm{std}}^{0}=t,
\qquad
\lambda_{t,\mathrm{std}}^{0}=\lambda_t.
\]
We use a dot to denote the first variation at $\eps=0$:
\begin{align}\label{eqn-dot-standard-notation}
\dot\lambda_{t,\mathrm{std}}
:=
\left.
\frac{\partial}{\partial\eps}
\lambda_{t,\mathrm{std}}^{\eps\nu}
\right|_{\eps=0},
\qquad
\dot a_{t,\mathrm{std}}
:=
\left.
\frac{\partial}{\partial\eps}
a_{t,\mathrm{std}}^{\eps\nu}
\right|_{\eps=0}.
\end{align}

The following theorem gives the first-variation formulas in this
standard normalization.  It is the basic result of the paper; the
formulas for the asymptotically normalized deformation will subsequently
be obtained from it by comparing the two normalizations.
\begin{theorem}\label{thm-main-normal}
Let $\nu\in L^\infty(\HH)$, and let
$\psi_{\mathrm{std}}^{\eps\nu}$,
$\lambda_{t,\mathrm{std}}^{\eps\nu}$, and
$a_{t,\mathrm{std}}^{\eps\nu}$ be as defined above.  Then, for every fixed $0\le t<T_+$, the maps
\[
\eps\longmapsto\lambda_{t,\mathrm{std}}^{\eps\nu},
\qquad
\eps\longmapsto a_{t,\mathrm{std}}^{\eps\nu}
\]
are differentiable at $\eps=0$, and
\begin{align}
\dot\lambda_{t,\mathrm{std}}
&=
-\frac{2}{\pi}\operatorname{Re}
\iint_{D_t}
\nu(z)\left(
\frac{g_t'(z)^2}{g_t(z)-\lambda_t}
-\frac1z
-\frac{\lambda_t}{z(z-1)}
\right)\dA(z)
\nonumber\\
&\quad
+\frac{2}{\pi}\operatorname{Re}
\iint_{K_t^\circ}
\nu(z)\left(
\frac1z
+\frac{\lambda_t}{z(z-1)}
\right)\dA(z),
\label{eqn-main-norm-1}
\\
\dot a_{t,\mathrm{std}}
&=
\frac{1}{\pi}\operatorname{Re}
\iint_{D_t}
\nu(z)\left(
g_t'(z)^2-1+\frac{4t}{z(z-1)}
\right)\dA(z)
\nonumber\\
&\quad
+\frac{1}{\pi}\operatorname{Re}
\iint_{K_t^\circ}
\nu(z)\left(
-1+\frac{4t}{z(z-1)}
\right)\dA(z).
\label{eqn-main-norm-2}
\end{align}
All four integrals are absolutely convergent.
\end{theorem}
\begin{reviewchange}
The notation $K_t^\circ=K_t\cap\HH$ denotes the portion of the
trace in the upper half-plane, not its topological interior. Jordan
arcs may have positive planar area \cite{Os03}; hence the trace
integrals cannot in general be omitted. They vanish if
$\Area(K_t^\circ)=0$ or $\nu=0$ almost everywhere on $K_t^\circ$.
The precise zero-extension convention and a test showing the necessity
of these terms are given in \S\ref{sec-trace-area}.
\end{reviewchange}
We next turn to a normalization adapted to the behavior at infinity.
As in \cite{SW24}, from the Loewner-theoretic point of view one would like the
quasiconformal deformation itself to satisfy a hydrodynamic-type
normalization at infinity, or at least to be asymptotic to the identity
up to an additive constant.  In the compactly supported setting this is
precisely the normalization used in \textbf{Theorem~A}, namely
\[
\psi^{\eps\nu}(z)-z=O(1),
\qquad z\to\infty.
\]
As observed above, however, such a normalization need not exist for a
general bounded Beltrami differential which does not decay at infinity.

Motivated by the example in \S~\ref{sec-motiv-exam},
 there is a
natural substitute for the oscillatory coefficients.  Suppose that the symmetric extension of $\nu$ to
$\C$ has the form
\[
\nu(z)
=
a\Bigl(\frac{\bar z}{z}\Bigr)^{n-1}
+\nu_0(z),
\]
where the first term is the non-decaying part and $\nu_0$ has sufficient
integrability at infinity.  The Beltrami equation corresponding to the
leading term is modeled by the explicit quasiconformal map
\[
z\longmapsto
\bigl(z^n+a\bar z^{\,n}\bigr)^{1/n}.
\]
It is therefore natural to require the quasiconformal deformation to
differ from this model only by an additive constant and an $o(1)$ term
at infinity.  We refer to this as the \emph{asymptotic normalization}.

For sufficiently small $|\eps|$, let $\phi^{\eps\nu}$ denote the
asymptotically normalized quasiconformal self-map of $\HH$ with Beltrami
coefficient $\eps\nu$, whose existence and uniqueness are established
in Theorem~\ref{thm-main-hh} below.  
Thus
\[
\phi^{\eps\nu}(0)=0
\qquad\text{and}\qquad
\phi^{\eps\nu}(z)
=
\bigl(z^n+\eps a\bar z^{\,n}\bigr)^{1/n}
+B_\eps+o(1),
\qquad z\to\infty.
\]

We use the subscript $\mathrm{asy}$ to distinguish the Loewner
quantities associated with this asymptotic normalization from those
associated with the standard normalization above.  For each fixed
$\review{0\le t<T_+}$, set
\[
K_{t,\mathrm{asy}}^{\eps\nu}
:=
\phi^{\eps\nu}(K_t)
=
\phi^{\eps\nu}\bigl(\gamma[0,t]\bigr).
\]
Let
\[
g_{t,\mathrm{asy}}^{\eps\nu}\colon
\HH\setminus K_{t,\mathrm{asy}}^{\eps\nu}
\longrightarrow\HH
\]
be the hydrodynamically normalized conformal map,
\begin{align}\label{eqn-asy-loewner-map}
g_{t,\mathrm{asy}}^{\eps\nu}(z)
=
z+\frac{2a_{t,\mathrm{asy}}^{\eps\nu}}{z}
+O\!\left(\frac1{z^2}\right),
\qquad z\to\infty.
\end{align}
Thus 
\[
2a_{t,\mathrm{asy}}^{\eps\nu}
=
\operatorname{hcap}
\bigl(K_{t,\mathrm{asy}}^{\eps\nu}\bigr).
\]
The corresponding driving function is
\begin{align}\label{eqn-asy-driving-point}
\lambda_{t,\mathrm{asy}}^{\eps\nu}
=
g_{t,\mathrm{asy}}^{\eps\nu}
\bigl(\phi^{\eps\nu}(\gamma(t))\bigr).
\end{align}
We write (if these derivatives exist)
\begin{align}\label{eqn-asy-dot-notation}
\dot\lambda_{t,\mathrm{asy}}
:=
\left.
\frac{\partial}{\partial\eps}
\lambda_{t,\mathrm{asy}}^{\eps\nu}
\right|_{\eps=0},
\qquad
\dot a_{t,\mathrm{asy}}
:=
\left.
\frac{\partial}{\partial\eps}
a_{t,\mathrm{asy}}^{\eps\nu}
\right|_{\eps=0}.
\end{align}

The next theorem first establishes the existence and uniqueness of the
asymptotically normalized quasiconformal map and then gives the
corresponding first-variation formulas.

\begin{theorem}\label{thm-main-hh}
Suppose that $\nu\in L^\infty(\C)$ is symmetric with respect to the real
axis, satisfies $\|\nu\|_\infty<1$, and has the form
\begin{align}\label{eqn-asy-class}
\nu(z)
=
a\,\frac{\bar z^{\,n-1}}{z^{n-1}}
+\nu_0(z),
\end{align}
where $a\in(-1,1)$, $n\ge1$ is an integer, and
$\nu_0\in L^r(\C)$ for some $0<r<2$.
Then the following hold.

\begin{enumerate}
\item
There exists a unique quasiconformal map
$\phi^\nu\colon\HH\to\HH$ satisfying
\[
\partial_{\bar z}\phi^\nu
=
\nu\,\partial_z\phi^\nu,
\qquad
\phi^\nu(0)=0,
\]
and
\begin{align}\label{eqn-asy-normalization}
\phi^\nu(z)
=
\bigl(z^n+a\bar z^{\,n}\bigr)^{1/n}
+B+o(1),
\qquad z\to\infty.
\end{align}

\item
For $\eps\in\R$ with $|\eps|$ sufficiently small, let
$\phi^{\eps\nu}$ be the asymptotically normalized solution associated
with $\eps\nu$, and let
$\lambda_{t,\mathrm{asy}}^{\eps\nu}$ and
$a_{t,\mathrm{asy}}^{\eps\nu}$ be defined by
\eqref{eqn-asy-loewner-map}--\eqref{eqn-asy-driving-point}.  Then
\begin{align*}
\dot\lambda_{t,\mathrm{asy}}
&=
-\frac2\pi
\PV\!\operatorname{Re}\iint_{D_t}
\nu(z)\left(
\frac{g_t'(z)^2}{g_t(z)-\lambda_t}-\frac1z
\right)\dA(z)
\\
&\quad
+\frac2\pi\operatorname{Re}\iint_{K_t^\circ}
\frac{\nu(z)}{z}\dA(z),
\\
\dot a_{t,\mathrm{asy}}
&=
\frac1\pi
\PV\operatorname{Re}\iint_{D_t}
\nu(z)\bigl(g_t'(z)^2-1\bigr)\dA(z)
\\
&\quad
-\frac1\pi\operatorname{Re}\iint_{K_t^\circ}
\nu(z)\dA(z).
\end{align*}
Here $\PV$ denotes the circular principal value at infinity, whereas
the integrals over $K_t^\circ$ are ordinary absolutely convergent
integrals.
\end{enumerate}
\end{theorem}

The existence part is proved in
Theorem~\ref{thm-yibande-n-new}.  The variational part is obtained
in Theorem~\ref{thm-universal-n} from Theorem~\ref{thm-main-normal} by
the normalization-transfer identity, rather than by a second
calculation of the Loewner coefficients.

\begin{remark}\label{rem-n-equals-one}
When $n=1$, the leading Beltrami coefficient is constant, corresponding to an affine quasiconformal deformation; in particular, the theorem
includes the global deformation $z\mapsto z+\eps a\bar z$.  For $n>1$
and $a\ne0$, the leading coefficient has a non-decaying angular oscillation.
\end{remark}

The paper is organized in the order dictated by the transfer argument.
Section~\ref{sec-standard-normalization} establishes the variational formulas under the standard normalization, culminating in Theorem~\ref{thm-main-normal}. 
Section~\ref{sec-existence-oscillatory} develops the analytic preliminaries, constructs the asymptotically normalized quasiconformal map, and derives the oscillatory variational formulas via the normalization-transfer identity; see Theorem~\ref{thm-main-hh} and Theorem~\ref{thm-universal-n}. 
The explicit contour computations for the straight-slit example are deferred to Appendix~\ref{sec-appendix-contours}.

\section{The general theorem under standard normalization}
\label{sec-standard-normalization}
This section is devoted to the proof of  Theorem~\ref{thm-main-normal}. We shall use the notation introduced in \S~\ref{sec-main-sdh}   throughout this
section.  In particular,
\[
K_t=\gamma[0,t],\qquad
K_t^\circ=K_t\cap\HH,\qquad
D_t=\HH\setminus K_t,
\]
and
\[
g_t\colon D_t\to\HH,
\qquad
g_t(z)=z+\frac{2t}{z}+O(z^{-2}),
\qquad z\to\infty.
\]
For $\nu\in L^\infty(\HH)$, let
$\psi_{\mathrm{std}}^{\eps\nu}$ be the standardly normalized
quasiconformal deformation fixing $0,1,\infty$, and let
$\lambda_{t,\mathrm{std}}^{\eps\nu}$ and
$2a_{t,\mathrm{std}}^{\eps\nu}$ denote the driving function
and half-plane capacity of $K_{t,\mathrm{std}}^{\eps\nu}=\psi_{\mathrm{std}}^{\eps\nu}(\gamma[0,t])$.

\subsection{The normalized Ahlfors--Bers variation}
For a symmetrically extended $\mu\in L^\infty(\C)$, let
$w^{\eps\mu}$ be the solution of the Beltrami equation fixing
$0,1,\infty$. The Ahlfors--Bers theorem \cite{AB60}
(see also \cite[Theorem~4.37]{IT92}) gives
\begin{align}\label{eqn-AB-normalized-expansion}
w^{\eps\mu}(z)
=z+\eps V_\mu(z)+o(\eps)
\end{align}
locally uniformly on $\C$.
For $z\in\C\setminus\{0,1\}$, the variation field is given by
\begin{align}\label{eqn-psi-variation-integral}
V_\mu(z)
=-\frac{z(z-1)}{\pi}
\iint_{\C}
\frac{\mu(\zeta)}
{\zeta(\zeta-1)(\zeta-z)}\dA(\zeta).
\end{align}
For each such $z$, the integral is absolutely convergent,
and its kernel is $O(|\zeta|^{-3})$ as $|\zeta|\to\infty$.
The function $V_\mu$ extends continuously to $0$ and $1$,
with $V_\mu(0)=V_\mu(1)=0$.
%
%
%

\begin{lemma}\label{lemma-inverse-variation}
Under the assumptions above,
\begin{equation}\label{eqn-inverse-variation-standard}
(w^{\eps\mu})^{-1}(z)
=
z-\eps V_\mu(z)+o(\eps)
\end{equation}
locally uniformly on $\C$.
\end{lemma}
\begin{proof}
Since $w^{\eps\mu}$ are homeomorphisms converging locally uniformly
to $\mathrm{id}$, their inverses also converge locally uniformly
to $\mathrm{id}$. Fix a compact set $E\subset\C$ and write
$z_\eps=(w^{\eps\mu})^{-1}(z)$. For sufficiently small $|\eps|$,
all points $z_\eps$, with $z\in E$, lie in a fixed compact set
$E'\subset\C$. Applying \eqref{eqn-AB-normalized-expansion}
uniformly on $E'$, we obtain
\[
z=w^{\eps\mu}(z_\eps)
=z_\eps+\eps V_\mu(z_\eps)+o(\eps)
=z_\eps+\eps V_\mu(z)+o(\eps),
\]
uniformly for $z\in E$, where the last equality follows from
$z_\eps\to z$ uniformly on $E$ and the uniform continuity of
$V_\mu$ on $E\cup E'$. Rearranging proves the assertion.
\end{proof}
\begin{lemma}\label{lemma-AB-Laurent-coefficients}
If $\mu$ has compact support, then
\begin{align}\label{eqn-AB-Laurent-expansion}
V_\mu(z)
=
\frac{P_\mu}{\pi}z+\frac{Q_\mu}{\pi}
+\frac{M_\mu}{\pi z}+O(z^{-2}),
\qquad z\to\infty,
\end{align}
where 
\begin{align*}
P_\mu:=\iint_{\C}\frac{\mu(\zeta)}{\zeta(\zeta-1)}\dA (\zeta),
\quad
Q_\mu:=\iint_{\C}\frac{\mu(\zeta)}{\zeta}\dA (\zeta),
\quad
M_\mu:=\iint_{\C}\mu(\zeta)\dA (\zeta).
\end{align*}
\end{lemma}
\begin{proof}
Using
\[
-\frac{z(z-1)}{\zeta(\zeta-1)(\zeta-z)}
=
\frac{z}{\zeta(\zeta-1)}
+\frac1\zeta
-\frac1{\zeta-z},
\]
we have
\[
V_\mu(z)
=
\frac{P_\mu}{\pi}z
+\frac{Q_\mu}{\pi}
-\frac1\pi\iint_{\C}
\frac{\mu(\zeta)}{\zeta-z}\,dA(\zeta).
\]
Since $\mu$ has compact support,
\[
-\frac1{\zeta-z}
=
\frac1z+O(z^{-2}),
\qquad z\to\infty,
\]
uniformly for $\zeta\in\operatorname{supp}\mu$. Therefore
\[
-\frac1\pi\iint_{\C}
\frac{\mu(\zeta)}{\zeta-z}\,dA(\zeta)
=
\frac{M_\mu}{\pi z}+O(z^{-2}),
\]
which gives \eqref{eqn-AB-Laurent-expansion}.
\end{proof}

\subsection{An auxiliary conformal mapping}
Fix $\review{0\le t<T_+}$, write $G_t=g_t^{-1}$, and define
\begin{align}\label{eqn-nu-t-normal}
\nu_t(w)
=
\nu(G_t(w))
\frac{\overline{G_t'(w)}}{G_t'(w)},
\qquad w\in\HH .
\end{align}
Notice that $\nu_t$ only depends on  the restriction of $\nu$ to $D_t$.  Extend
$\nu_t$ symmetrically to $\C$, and let $\Psi_t^\eps$ be the normalized
quasiconformal map with Beltrami coefficient $\eps\nu_t$, fixing
$0,1,\infty$. Similar to the previous case, $\Psi_t^\eps$ preserves $\R$ and $\HH$.
\begin{lemma}\label{lemma-standard-Stoilow}
There exists  a conformal mapping
\[
H_t^\eps\colon
\HH\setminus K_{t,\mathrm{std}}^{\eps\nu}\longrightarrow\HH
\]
such that
\begin{align}\label{eqn-H-factorization-normal}
\Psi_t^\eps
=
H_t^\eps\circ\psi_{\mathrm{std}}^{\eps\nu}\circ G_t,
\end{align}
or equivalently
\[
H_t^\eps
=
\Psi_t^\eps\circ g_t\circ
(\psi_{\mathrm{std}}^{\eps\nu})^{-1}.
\]
It has a Laurent expansion
\begin{align}\label{eqn-H-Laurent-normal}
H_t^\eps(z)
=
B_{-1}(\eps)z+B_0(\eps)
+\frac{B_1(\eps)}{z}+O(z^{-2}),
\end{align}
where $B_{-1}(\eps)>0$ and $B_0(\eps)\in\R$.  Moreover,
\begin{align}
g_{t,\mathrm{std}}^{\eps\nu}(z)
&=
\frac{H_t^\eps(z)-B_0(\eps)}{B_{-1}(\eps)},
\label{eqn-g-epsilon-H-normal}\\
a_{t,\mathrm{std}}^{\eps\nu}
&=
\frac{B_1(\eps)}{2B_{-1}(\eps)},
\label{eqn-capacity-B-normal}\\
\lambda_{t,\mathrm{std}}^{\eps\nu}
&=
\frac{\Psi_t^\eps(\lambda_t)-B_0(\eps)}{B_{-1}(\eps)}.
\label{eqn-driving-Psi-normal}
\end{align}
\end{lemma}

\begin{proof}
The maps $\psi_{\mathrm{std}}^{\eps\nu}\circ G_t$ and $
\Psi_t^\eps$
have the same Beltrami coefficient $\eps\nu_t$.  The Stoilow
factorization theorem therefore gives
\eqref{eqn-H-factorization-normal}.  All mappings  fix infinity, so
$H_t^\eps(\infty)=\infty$.  Schwarz reflection  extends $H_t^\eps$ conformally  to a neighborhood of infinity, which gives 
\eqref{eqn-H-Laurent-normal}.  By symmetry  and preservation of the upper half-plane, we have $B_0(\varepsilon)\in\mathbb{R}$  and 
$B_{-1}(\eps)>0$. Composing with the  real affine map
\[
I_\eps(w)=\frac{w-B_0(\eps)}{B_{-1}(\eps)}
\]
gives
\eqref{eqn-g-epsilon-H-normal}. Comparison of the  coefficient of $z^{-1}$
 gives \eqref{eqn-capacity-B-normal}.  Finally, taking the
boundary limit as $w\to\lambda_t$ from $\HH$ in
\eqref{eqn-H-factorization-normal} gives
\[
H_t^\eps\bigl(\psi_{\mathrm{std}}^{\eps\nu}(\gamma(t))\bigr)
=
\Psi_t^\eps(\lambda_t).
\]
Applying $I_\eps$ gives  \eqref{eqn-driving-Psi-normal}.
\end{proof}
Recall the symmetric extension in \eqref{eqn-symmetric-extension}, and
write
\begin{align}\label{eqn-vectorfield}
V:=V_{\widehat\nu},
\qquad
V_t:=V_{\nu_t}.
\end{align}
Define, for $z\in D_t$,
\begin{equation}\label{eqn-Wt-definition}
W_t(z)
:=
V_t(g_t(z))-g_t'(z)V(z).
\end{equation}
By symmetry,
$V(\bar z)=\overline{V(z)}$ and
$V_t(\bar z)=\overline{V_t(z)}$; hence, whenever $g_t$ is extended
across the real axis by Schwarz reflection, $W_t$ has the reflected
extension
\[
W_t(\bar z)=\overline{W_t(z)}.
\]
We continue to denote this extension by $W_t$.
\begin{lemma}\label{lemma-common-annulus}
There exist $\rho>0$, $\delta>0$, and $\eps_0>0$ such that, for
\[
A_{\rho,\delta}
:=
\{z\in\C:\rho-\delta<|z|<\rho+\delta\},
\]
the maps $g_t$ and $H_t^\eps$ admit holomorphic extensions to
an open neighborhood of $\overline{A_{\rho,\delta}}$
whenever $|\eps|<\eps_0$.  Moreover,
\begin{equation}\label{eqn-H-first-variation-annulus}
\lim_{\eps\to0}
\sup_{z\in \overline{A_{\rho,\delta}}}
\left|
\frac{H_t^\eps(z)-g_t(z)}{\eps}
-W_t(z)
\right|
=0.
\end{equation}
For every compact set $E\subset D_t$, one also has
\begin{equation}\label{eqn-H-first-variation-local}
\lim_{\eps\to0}
\sup_{z\in E}
\left|
\frac{H_t^\eps(z)-g_t(z)}{\eps}
-W_t(z)
\right|
=0.
\end{equation}
Therefore, if
\[
H_t^\eps(z)
=
B_{-1}(\eps)z+B_0(\eps)
+\frac{B_1(\eps)}{z}+O(z^{-2})
\]
is the Laurent expansion at infinity, then
$B_{-1},B_0,B_1$ are differentiable at $\eps=0$, and
\begin{equation}\label{eqn-B-derivative-contour}
\dot B_j(0)
=
\frac{1}{2\pi i}
\int_{|z|=\rho}
W_t(z)z^{j-1}\,dz,
\qquad j=-1,0,1.
\end{equation}
Consequently,
$a_{t,\mathrm{std}}^{\eps\nu}$ and
$\lambda_{t,\mathrm{std}}^{\eps\nu}$ are differentiable at
$\eps=0$.
\end{lemma}

\begin{proof}
Choose $R>0$ such that
\[
K_t\subset\{|z|<R\}.
\]
Since $\psi_{\mathrm{std}}^{\eps\nu}\longrightarrow\operatorname{id}$
uniformly on $K_t$, after decreasing $\eps_0$ if necessary, we have 
\[
K_{t,\mathrm{std}}^{\eps\nu}
\subset\{|z|<R\},
\qquad |\eps|<\eps_0.
\]
Choose $\rho>R$ and $\delta>0$ so that $\rho-\delta>R$.
Thus the annulus $A_{\rho,\delta}$ does not intersect with  either $K_t$ or
$K_{t,\mathrm{std}}^{\eps\nu}$ for $|\eps|<\eps_0$.
Schwarz reflection gives holomorphic
extensions of both maps to $A_{\rho,\delta}$.
In fact both maps extend to the common exterior region
$\{|z|>R\}$, which contains an open neighborhood of
$\overline{A_{\rho,\delta}}$.

We next prove the first-order expansion.  By
Lemma~\ref{lemma-inverse-variation},
\begin{equation}\label{eqn-inverse-use-common-annulus}
(\psi_{\mathrm{std}}^{\eps\nu})^{-1}(z)
=
z-\eps V(z)+o(\eps)
\end{equation}
uniformly on $\overline{A_{\rho,\delta}}$.
Since $g_t$ is holomorphic on a slightly larger annulus, Taylor's
formula gives
\begin{align}
g_t\!\left(
(\psi_{\mathrm{std}}^{\eps\nu})^{-1}(z)
\right)
&=
g_t(z)-\eps g_t'(z)V(z)+o(\eps)\qquad \text{uniformly on $\overline{A_{\rho,\delta}}$.}
\label{eqn-gt-composition-variation}
\end{align}
Recall that $H_t^\eps
=
\Psi_t^\eps
\circ g_t
\circ
(\psi_{\mathrm{std}}^{\eps\nu})^{-1}$.
The Ahlfors--Bers expansion for $\Psi_t^\eps$ gives
\[
\Psi_t^\eps(w)
=
w+\eps V_t(w)+o(\eps)
\]
uniformly on the compact sets under consideration.
Put
$w_\eps(z)=g_t((\psi_{\mathrm{std}}^{\eps\nu})^{-1}(z))$.
These points stay in a single compact set and converge uniformly to
$g_t(z)$ on the closed annulus. Uniform continuity of $V_t$ on that
compact set therefore gives
$V_t(w_\eps(z))=V_t(g_t(z))+o(1)$ uniformly; no differentiability of
$V_t$ is required.
Combining this with \eqref{eqn-gt-composition-variation}, we obtain
\begin{align*}
H_t^\eps(z)
=
g_t(z)
-\eps g_t'(z)V(z)
+\eps V_t(g_t(z))
+o(\eps)=
g_t(z)+\eps W_t(z)+o(\eps),\qquad \text{uniformly on $\overline{A_{\rho,\delta}}$.  }
\end{align*}
This proves
\eqref{eqn-H-first-variation-annulus}.  The same argument, with
$\overline{A_{\rho,\delta}}$ replaced by an arbitrary compact set
$E\subset D_t$, proves \eqref{eqn-H-first-variation-local}.

For $j=-1,0,1$, Cauchy's formula gives
\[
B_j(\eps)
=
\frac{1}{2\pi i}
\int_{|z|=\rho}
H_t^\eps(z)z^{j-1}\,dz.
\]
Using \eqref{eqn-H-first-variation-annulus} in this identity gives
\[
\frac{B_j(\eps)-B_j(0)}{\eps}
\longrightarrow
\frac{1}{2\pi i}
\int_{|z|=\rho}
W_t(z)z^{j-1}\,dz,
\]
which proves \eqref{eqn-B-derivative-contour}.

Finally, the Ahlfors--Bers expansion shows that
$\Psi_t^\eps(\lambda_t)$ is differentiable at $\eps=0$.
Together with the differentiability of $B_{-1},B_0,B_1$,
\eqref{eqn-capacity-B-normal} and \eqref{eqn-driving-Psi-normal}
therefore imply that $a_{t,\mathrm{std}}^{\eps\nu}$ and
$\lambda_{t,\mathrm{std}}^{\eps\nu}$ are differentiable at $\eps=0$.
\end{proof}

\subsection{Proof of Theorem~\ref{thm-main-normal} for compactly  supported   Beltrami coefficients}\label{sec-bounded}
In this section we always 
assume  that $\widehat\nu$ has compact support.  Then
$\nu_t$ also has compact support.  Write
\[
P=P_{\widehat\nu},\quad Q=Q_{\widehat\nu},\quad M=M_{\widehat\nu},
\qquad
P_t=P_{\nu_t},\quad Q_t=Q_{\nu_t},\quad M_t=M_{\nu_t}.
\]
Using
\[
g_t(z)=z+\frac{2t}{z}+O(z^{-2}),
\qquad
g_t'(z)=1-\frac{2t}{z^2}+O(z^{-3}),
\]
Lemma~\ref{lemma-AB-Laurent-coefficients} gives
\begin{align*}
V_t(g_t(z))
&=
\frac{P_t}{\pi}z+\frac{Q_t}{\pi}
+\frac{2tP_t+M_t}{\pi z}+O(z^{-2}),
\\
g_t'(z)V(z)
&=
\frac{P}{\pi}z+\frac{Q}{\pi}
+\frac{M-2tP}{\pi z}+O(z^{-2}).
\end{align*}
Since $\widehat\nu$ and $\nu_t$ have compact support, the corresponding
variation fields $V$ and $V_t$ are holomorphic in a neighborhood of
infinity. Hence
\[
W_t(z):=V_t(g_t(z))-g_t'(z)V(z)
\]
is also holomorphic in a neighborhood of infinity. Therefore, the
preceding expansions give the Laurent expansion
\[
W_t(z)
=
\frac{P_t-P}{\pi}z
+\frac{Q_t-Q}{\pi}
+\frac{2t(P_t+P)+M_t-M}{\pi z}
+O(z^{-2}).
\]
By Lemma~\ref{lemma-common-annulus}, the coefficients of this Laurent
expansion are precisely the first variations of the corresponding
Laurent coefficients of $H_t^\varepsilon$. Thus
\begin{align}\label{eqn-dot-Bminusone-compact}
\dot B_{-1}(0)&=\frac{P_t-P}{\pi},\quad
\dot B_0(0)=\frac{Q_t-Q}{\pi},\quad
\dot B_1(0)=\frac{2t(P_t+P)+M_t-M}{\pi}.
\end{align}
Differentiating \eqref{eqn-driving-Psi-normal} gives
\[
\dot\lambda_{t,\mathrm{std}}
=
V_t(\lambda_t)-\dot B_0(0)-\lambda_t\dot B_{-1}(0).
\]
The partial-fraction identity used in
Lemma~\ref{lemma-AB-Laurent-coefficients}, now evaluated at
$\lambda_t$, and
\eqref{eqn-dot-Bminusone-compact} gives
\begin{align}\label{eqn-driving-before-domain-change}
\dot\lambda_{t,\mathrm{std}}
=
-\frac1\pi
\iint_{\C}
\frac{\nu_t(w)}{w-\lambda_t}\dA (w)
+\frac1\pi
\iint_{\C}\frac{\widehat\nu(z)}{z}\dA (z)
+
\frac{\lambda_t}{\pi}
\iint_{\C}
\frac{\widehat\nu(z)}{z(z-1)}\dA (z).
\end{align}
All three integrals are absolutely convergent in the bounded-support
case.
No exclusion of $\lambda_t=0$ or $\lambda_t=1$ is needed here.
Indeed, for $x\in\R$ the partial-fraction identity can be written as
\[
V_t(x)=-\frac1\pi\iint_{\C}\frac{\nu_t(w)}{w-x}\dA(w)
       +\frac{Q_t}{\pi}+\frac{xP_t}{\pi}.
\]
The Cauchy integral on the right is continuous in $x$ for a bounded
compactly supported coefficient, so the identity extends to $x=0,1$.
Substitution into the preceding expression for
$\dot\lambda_{t,\mathrm{std}}$ gives the same formula at those values.
From \eqref{eqn-nu-t-normal}, with $w=g_t(z)$,
\[
\nu_t(g_t(z))
=
\nu(z)\frac{g_t'(z)}{\overline{g_t'(z)}},
\qquad
dA(g_t(z))=|g_t'(z)|^2dA(z).
\]
Consequently,
\begin{align*}
\iint_{\HH}
\frac{\nu_t(w)}{w-\lambda_t}\dA (w)
=
\iint_{D_t}
\nu(z)\frac{g_t'(z)^2}{g_t(z)-\lambda_t}\dA (z).
\end{align*}
Using symmetry in \eqref{eqn-driving-before-domain-change} and $\HH=D_t\cup K_t^\circ$ gives
\begin{align*}
\dot\lambda_{t,\mathrm{std}}
&=
-\frac{2}{\pi}\operatorname{Re}
\iint_{D_t}
\nu(z)\frac{g_t'(z)^2}{g_t(z)-\lambda_t}\dA(z)
+\frac{2}{\pi}\operatorname{Re}
\iint_{\HH}
\nu(z)\left(
\frac1z+\frac{\lambda_t}{z(z-1)}
\right)\dA(z)
\\
&=
-\frac{2}{\pi}\operatorname{Re}
\iint_{D_t}
\nu(z)\left(
\frac{g_t'(z)^2}{g_t(z)-\lambda_t}
-\frac1z-\frac{\lambda_t}{z(z-1)}
\right)\dA(z)
\\
&\quad
+\frac{2}{\pi}\operatorname{Re}
\iint_{K_t^\circ}
\nu(z)\left(
\frac1z+\frac{\lambda_t}{z(z-1)}
\right)\dA(z),
\end{align*}
which proves \eqref{eqn-main-norm-1} for $\nu$  that has compact support.

Similarly, differentiating \eqref{eqn-capacity-B-normal} and using
\eqref{eqn-dot-Bminusone-compact}
gives
\begin{align*}
2\dot a_{t,\mathrm{std}}
=
\frac{4t}{\pi}
\iint_{\C}
\frac{\widehat\nu(z)}{z(z-1)}\dA (z)
+
\frac1\pi\iint_{\C}\nu_t(w)\dA (w)
-\frac1\pi\iint_{\C}\widehat\nu(z)\dA (z).
\end{align*}
A change of variable  gives
\begin{align*}
\iint_{\HH}\nu_t(w)\dA (w)
=
\iint_{D_t}\nu(z)g_t'(z)^2\dA (z).
\end{align*}
Taking upper-half-plane real parts, dividing by $2$, and using
$\HH=D_t\cup K_t^\circ$ yields
\begin{align*}
\dot a_{t,\mathrm{std}}
&=
\frac1\pi\operatorname{Re}
\iint_{D_t}
\nu(z)\left(
g_t'(z)^2-1+\frac{4t}{z(z-1)}
\right)\dA(z)
+\frac1\pi\operatorname{Re}
\iint_{K_t^\circ}
\nu(z)\left(
-1+\frac{4t}{z(z-1)}
\right)\dA(z).
\end{align*}
This proves \eqref{eqn-main-norm-2} when $\nu$ has compact support.

\subsection{Proof of Theorem~\ref{thm-main-normal} for arbitrary   Beltrami coefficients}

For use in the estimates below, we introduce the following notation which appeared in \eqref{eqn-main-norm-1} and \eqref{eqn-main-norm-2} in Theorem~\ref{thm-main-normal}:
for $z\in D_t=\mathbb{H}\setminus\gamma[0,t]$, set
\begin{align}\label{eqn-ABC_C}
A_t(z):=\frac{g_t'(z)^2}{g_t(z)-\lambda_t},\quad
B_t(z):=\frac1z+\frac{\lambda_t}{z(z-1)},\quad
C_t(z):=-1+\frac{4t}{z(z-1)}.
\end{align}
Then we have 
\begin{lemma}\label{lemma-standard-kernel-integrability}
The functions in \eqref{eqn-ABC_C} satisfy
\[
A_t-B_t\in L^1(D_t),\qquad
g_t'^2+C_t\in L^1(D_t),
\]
and
\[
B_t|_{K_t^\circ}\in L^1(K_t^\circ),\qquad
C_t|_{K_t^\circ}\in L^1(K_t^\circ).
\]
\end{lemma}

\begin{proof}
The hydrodynamic expansion gives
\[
A_t(z)
=
\frac1z+\frac{\lambda_t}{z^2}+O(|z|^{-3}),
\qquad
B_t(z)
=
\frac1z+\frac{\lambda_t}{z^2}+O(|z|^{-3}),
\]
and
\[
g_t'(z)^2-1=-\frac{4t}{z^2}+O(|z|^{-3}),
\qquad
\frac{4t}{z(z-1)}
=
\frac{4t}{z^2}+O(|z|^{-3}).
\]
Hence both domain kernels are $O(|z|^{-3})$ at infinity.

For local integrability, let $U\subset\HH$ be a bounded domain.  The
set $g_t(U\cap D_t)$ is bounded, because
$G_t(w)=g_t^{-1}(w)=w-2t/w+O(w^{-2})$ at infinity.  The area formula gives
\[
\iint_{U\cap D_t}|A_t(z)|\dA (z)
=
\iint_{g_t(U\cap D_t)}
\frac{dA(w)}{|w-\lambda_t|}<\infty
\]
and
\[
\iint_{U\cap D_t}|g_t'(z)|^2\dA (z)
=
\Area\bigl(g_t(U\cap D_t)\bigr)<\infty .
\]
\review{The rational terms $B_t$ and $C_t$ are also locally
area-integrable on $\HH$: their possible finite poles at $0$ and $1$
are simple. Finally, the only possible boundary singularity} of $B_t$ and $C_t$ on the closure of  $K_t^\circ$ is bounded by
a constant multiple of $|z|^{-1}$ near $0$.  This is locally
integrable, completing the proof.
\end{proof}

\review{For $R>0$, define the truncation:}
\begin{align*}
\widehat\nu_R(z):=\widehat\nu(z)\review{\mathds{1}_{\{|z|<R\}}(z)}.
\end{align*}
%
%
\begin{lemma}\label{lemma-standard-truncation-derivatives}
As $R\to\infty$,
\[
\dot\lambda_{t,\mathrm{std},R}
\longrightarrow
\dot\lambda_{t,\mathrm{std}},
\qquad
\dot a_{t,\mathrm{std},R}
\longrightarrow
\dot a_{t,\mathrm{std}} .
\]
\end{lemma}

\begin{proof}
Let
\[
V_R:=V_{\widehat\nu_R},\qquad V:=V_{\widehat\nu},
\]
and let $V_{t,R}$ and $V_t$ be the corresponding normalized vector
fields for $(\nu_R)_t$ and $\nu_t$.  Since the Ahlfors--Bers kernel
satisfies
\[
\left|
\frac{z(z-1)}
{\zeta(\zeta-1)(\zeta-z)}
\right|
\le C_E|\zeta|^{-3}
\]
for $z$ in a fixed compact set $E$ and $|\zeta|$ sufficiently large,
\[
V_R\longrightarrow V
\qquad\text{locally uniformly on }\C.
\]

Moreover,
\[
(\nu_R)_t(w)
=
\nu_t(w)\chi_{\{|G_t(w)|<R\}}.
\]
Since
\[
g_t(z)=z+\frac{2t}{z}+O(z^{-2}),
\qquad z\to\infty,
\]
we have $|g_t(z)|\ge |z|/2$ for all sufficiently large $|z|$.
Consequently, for all sufficiently large $R$,
\[
\operatorname{supp}\bigl((\nu_R)_t-\nu_t\bigr)
\subset\{|w|\ge R/2\}.
\]
The same kernel estimate therefore gives
\[
V_{t,R}\longrightarrow V_t
\qquad\text{locally uniformly on }\C.
\]

By Lemma~\ref{lemma-common-annulus},
\[
\dot H_{t,R}(z)
=
V_{t,R}(g_t(z))-g_t'(z)V_R(z),
\]
and hence $\dot H_{t,R}\to\dot H_t$ uniformly on the fixed circle
$|z|=\rho$.  It follows from
\eqref{eqn-B-derivative-contour} that
\[
\dot B_{j,R}(0)\longrightarrow\dot B_j(0),
\qquad j=-1,0,1.
\]
Also,
\[
V_{t,R}(\lambda_t)\longrightarrow V_t(\lambda_t).
\]
Differentiating
\eqref{eqn-capacity-B-normal} and
\eqref{eqn-driving-Psi-normal} now gives the desired limits.
\end{proof}

\begin{proof}[Proof of Theorem~\ref{thm-main-normal}]
The compact support case proved in \S\ref{sec-bounded} applied to $\nu_R$ gives
\eqref{eqn-main-norm-1} and \eqref{eqn-main-norm-2} with $\nu_R$.  By
Lemma~\ref{lemma-standard-kernel-integrability}, the right-hand sides
converge by dominated convergence.  By
Lemma~\ref{lemma-standard-truncation-derivatives}, the left-hand sides
converge to the first variations for $\nu$.  This completes the proof.
\end{proof}

\begin{reviewchange}
\subsection{Positive-area traces and the zero-extension convention}
\label{sec-trace-area}
The conformal area formula through $g_t$ integrates over $D_t$, not
over all of $\HH$. The remaining integrals over $K_t^\circ$ record
the part of the ambient Beltrami coefficient which is not seen by
the pullback coefficient $\nu_t$. We first record when these terms
vanish and then show that they can be essential.

Put $K:=\gamma([0,T_+))$.
\begin{corollary}\label{coro-standard-zero-extension}
Let $\nu^\circ\in L^\infty(\HH\setminus K)$ and define the Beltrami 
coefficient on the upper half-plane by
\begin{align}\label{eqn-global-zero-extension}
\nu(z):=
\begin{cases}
\nu^\circ(z),&z\in\HH\setminus K,\\
0,&z\in K\cap\HH.
\end{cases}
\end{align}
Then
\begin{align*}
\dot\lambda_{t,\mathrm{std}}
&=
-\frac{2}{\pi}\operatorname{Re}
\iint_{D_t}
\nu(z)\left(
\frac{g_t'(z)^2}{g_t(z)-\lambda_t}
-\frac1z-\frac{\lambda_t}{z(z-1)}
\right)\dA (z),
\\
\dot a_{t,\mathrm{std}}
&=
\frac{1}{\pi}\operatorname{Re}
\iint_{D_t}
\nu(z)\left(
g_t'(z)^2-1+\frac{4t}{z(z-1)}
\right)\dA (z).
\end{align*}
The same conclusion holds for a coefficient on all of $\HH$ whenever
$\nu=0$ almost everywhere on $K_t^\circ$, or whenever
$\Area(K_t^\circ)=0$.  
\end{corollary}

\begin{remark}
When $K$ has positive area, prescribing a Beltrami coefficient only on
$\HH\setminus K$ does not determine a normalized quasiconformal
self-map of $\HH$.  Formula~\eqref{eqn-global-zero-extension} is part
of the deformation.  The extension must be chosen once on the full
fixed trace $K$, rather than separately on $K_t$ for each $t$.
\end{remark}

\end{reviewchange}

\begin{proof}[Proof of Corollary~\ref{coro-standard-zero-extension}]
Since $K_t^\circ\subset K\cap\HH$, the zero extension
\eqref{eqn-global-zero-extension} vanishes on $K_t^\circ$, so both
trace integrals in Theorem~\ref{thm-main-normal} vanish.  If
$\Area(K_t^\circ)=0$, they vanish for every $L^\infty$-representative.
\end{proof}

\begin{reviewchange}
Notice that the  formulas  \eqref{eqn-main-norm-1} and \eqref{eqn-main-norm-2}  contain an integral over the trace $K_t^\circ$. When $\Area(K_t^\circ)=0$ this term vanishes, but when
$\Area(K_t^\circ)>0$ it may be nonzero.  The following
proposition shows that the trace contribution can be nonzero even when
the corresponding domain contribution vanishes.  Thus, in general,
the variation cannot be recovered from the Beltrami coefficient on
$D_t$ alone.

\begin{proposition}\label{prop-positive-area-test}
Suppose that $\Area(K_t^\circ)>0$.  Then there exists
$\nu\in L^\infty(\HH)$ supported on $K_t^\circ$ such that the
$D_t$-integral in \eqref{eqn-main-norm-2} vanishes, whereas
\[
\dot a_{t,\mathrm{std}}>0.
\]
\end{proposition}

\begin{proof}
Define
\[
\nu(z)=
\begin{cases}
\overline{C_t(z)}/|C_t(z)|,
& z\in K_t^\circ,\ C_t(z)\neq0,\\
0,&\text{otherwise}.
\end{cases}
\]
Then $\|\nu\|_\infty\le1$, and since $\nu=0$ almost everywhere on
$D_t$, the domain integral in \eqref{eqn-main-norm-2} is zero.  On the
other hand,
\[
\dot a_{t,\mathrm{std}}
=
\frac1\pi\iint_{K_t^\circ}|C_t(z)|\,\dA(z).
\]
Since $C_t$ is a nonzero rational function and therefore has only
finitely many zeros, while $\Area(K_t^\circ)>0$, the last integral is
strictly positive.
\end{proof}
\end{reviewchange}

%

\section{Proof of Theorem~\ref{thm-main-hh}}
\label{sec-existence-oscillatory}

\subsection{Analytic tools}
\label{sec-analytic-preliminaries}

Recall that the Cauchy transform and the regularized Cauchy transform
are defined by
\begin{align*}
C(h)(z)
&:=-\frac{1}{\pi}\iint_{\C}
\frac{h(\zeta)}{\zeta-z}\dA (\zeta),
\\
P(h)(z)
&:=-\frac{1}{\pi}\iint_{\C}
h(\zeta)\left(\frac{1}{\zeta-z}-\frac{1}{\zeta}\right)\dA (\zeta).
\end{align*}
The Beurling transform is defined as the singular integral
\begin{align*}
T(h)(z):=-\frac{1}{\pi}\lim_{\delta\downarrow0}
\iint_{|\zeta-z|>\delta}
\frac{h(\zeta)}{(\zeta-z)^2}\dA (\zeta),
\end{align*}
where the limit is understood in the usual singular-integral sense
(in $L^s$ and almost everywhere), or equivalently as the bounded
extension of this operator on $L^s(\C)$, $1<s<\infty$.
For more properties of these transforms, see \cite[Chapter~4]{AIM09}.
$T$ is bounded on $L^p(\C)$ for every $1<p<\infty$ and
$\lim_{p\to2}\|T\|_{L^p}=1$.  
The following lemma follows by the same argument as
\cite[Theorem~4.3.11]{AIM09}; the exponents $p$ and $q$ need not be
conjugate.

\begin{lemma}\label{thm-astal-modify}
Assume $1<q<2<p<\infty$.  If $h\in L^p(\C)\cap L^q(\C)$, then
\review{$C(h)$ is continuous on $\C$ and vanishes at infinity, i.e.}
\[
\lim_{z\to\infty}C(h)(z)=0 .
\]
Moreover, there exists a constant $M_{p,q}$ such that
\[
\|C(h)\|_\infty\le M_{p,q}\bigl(\|h\|_p+\|h\|_q\bigr).
\]
In fact, the absolute potential is bounded uniformly in $z$:
\begin{align}\label{eqn-cauchy-absolute-potential}
\sup_{z\in\C}\iint_{\C}\frac{|h(\zeta)|}{|\zeta-z|}\dA (\zeta)
\le
C_{p,q}\bigl(\|h\|_p+\|h\|_q\bigr).
\end{align}
\end{lemma}
\begin{reviewchange}
\begin{proof}
Split the potential into $|\zeta-z|<1$ and $|\zeta-z|\ge1$.
H\"older's inequality uses $|\cdot|^{-1}\in L^{p'}(B(0,1))$
because $p'<2$, and $|\cdot|^{-1}\in L^{q'}(\C\setminus B(0,1))$
because $q'>2$. This proves the uniform absolute-potential estimate
and the stated bound for $C(h)$. Approximate $h$ simultaneously in
$L^p$ and $L^q$ by smooth compactly supported functions. Their Cauchy
transforms are continuous and tend to zero at infinity, and the bound
just proved gives uniform convergence of the transforms. Thus
$C(h)\in C_0(\C)$.
\end{proof}
\end{reviewchange}

\subsection{The existence theorem with hydrodynamic
normalization}

\begin{theorem}\label{thm-QC-new-rd-1-modified}
Assume that the Beltrami differential $\mu$ satisfies
$\|\mu\|_\infty\le k<1$ and
$\mu\in L^r(\C)$ for some $0<r<2$.  Then there exists a unique quasiconformal map
$\psi^\mu\colon\C\to\C$ satisfying:
\begin{enumerate}
\item Beltrami equation:
$\partial_{\bar z}\psi^\mu=\mu\,\partial_z\psi^\mu$;
\item Hydrodynamic
normalization: $\psi^\mu(0)=0$, and as $z\to\infty$,
\[
\psi^\mu(z)=z+O(1);
\]
\item there exist  $1<q<2<p<\infty$ such that
$h=(I-\mu T)^{-1}\mu\in L^q(\C)\cap L^p(\C)$ and
\begin{align*}
\psi^\mu(z)=z+P(h)(z)=z+C(h)(z)-C(h)(0),
\end{align*}
where \review{$C(h)\in C_0(\C)$}.  Moreover, there exists a constant
$B$ such that $\psi^\mu(z)=z+B+o(1)$.
\end{enumerate}
\end{theorem}

\begin{proof}
Since $\mu\in L^r(\C)\cap L^\infty(\C)$ for some $0<r<2$, we may
choose
\[
\max\{1,r\}<q<2<p<\infty
\]
sufficiently close to $2$ so that
\[
k\|T\|_{L^q}<1,
\qquad
k\|T\|_{L^p}<1 .
\]
Indeed, the Beurling transform $T$ is bounded on $L^s(\C)$ for every
$1<s<\infty$, $\|T\|_{L^2}=1$, and $s\mapsto\|T\|_{L^s}$ is
continuous at $s=2$.  Moreover, since
\[
\|\mu\|_s^s
\le
\|\mu\|_\infty^{\,s-r}\|\mu\|_r^r,
\qquad s\ge r,
\]
we have $\mu\in L^q(\C)\cap L^p(\C)$.

Set
\[
X:=L^q(\C)\cap L^p(\C),
\qquad
\|f\|_X:=\|f\|_q+\|f\|_p .
\]
Then $X$ is a Banach space, and for every $f\in X$,
\[
\|\mu Tf\|_X
\le
\rho\|f\|_X,
\qquad
\rho:=
\max\bigl\{k\|T\|_{L^q},\,k\|T\|_{L^p}\bigr\}<1 .
\]
Hence $I-\mu T$ is invertible on $X$ by the Neumann series.  Define
\begin{align*}
h:=(I-\mu T)^{-1}\mu
=
\sum_{j=0}^{\infty}(\mu T)^j\mu .
\end{align*}
Then $h\in L^q(\C)\cap L^p(\C)$ and
\begin{align}\label{eq:h-beltrami}
h=\mu\bigl(1+T(h)\bigr).
\end{align}

Define
\[
\psi^\mu(z)
:=
z+P(h)(z)
=
z+C(h)(z)-C(h)(0).
\]
Then $\psi^\mu(0)=0$.  By Lemma~\ref{thm-astal-modify}, $C(h)$ is
continuous on $\C$ and
\[
C(h)(z)\longrightarrow0
\qquad\text{as }|z|\to\infty .
\]
Using $\partial_{\bar z}C(h)=h$ and $\partial_zC(h)=T(h)$ in the sense
of distributions, we obtain
\[
\partial_{\bar z}\psi^\mu=h,
\qquad
\partial_z\psi^\mu=1+T(h).
\]
Thus \eqref{eq:h-beltrami} gives
\[
\partial_{\bar z}\psi^\mu
=
\mu\,\partial_z\psi^\mu
\qquad\text{a.e.\ on }\C .
\]
Furthermore,
\begin{align*}
\psi^\mu(z)
=
z-C(h)(0)+o(1)
=
z+B+o(1),
\qquad |z|\to\infty ,
\end{align*}
where
\begin{align*}
B:=-C(h)(0)
=
\frac1\pi
\iint_{\C}
\frac{h(\zeta)}{\zeta}\dA (\zeta).
\end{align*}

It remains to prove that $\psi^\mu$ is a global homeomorphism.  Since
$p>2$ and $h,T(h)\in L^p(\C)$, we have $\psi^\mu\in W_{\mathrm{loc}}^{1,p}(\C)$.
Let $F^\mu\colon\Ch\longrightarrow\Ch$ 
be the normalized quasiconformal homeomorphism with Beltrami
coefficient $\mu$, fixing $0,1,\infty$.  Since $\psi^\mu$ and $F^\mu$
solve the same Beltrami equation,  there exists an entire function $\Phi$ such that
$\psi^\mu=\Phi\circ F^\mu$ .
Since $F^\mu$ fixes $\infty$,
we have 
\[
|\Phi(w)|
=
\left|
\psi^\mu\bigl((F^\mu)^{-1}(w)\bigr)
\right|
\longrightarrow\infty
\qquad\text{as }|w|\to\infty .
\]
Thus $\infty$ is a pole of $\Phi$, and hence $\Phi$ is a nonconstant
polynomial.

We claim that $\deg\Phi=1$.  Write
\begin{align*}
\Phi(w)-B
=
c\prod_{j=1}^{N}(w-a_j)^{m_j},
\qquad c\neq0,
\end{align*}
where $a_1,\ldots,a_N$ are the distinct zeros of $\Phi-B$, and $m_j$
is the multiplicity of $a_j$.  Since
\[
\psi^\mu(z)-B=z+o(1)
\qquad\text{as }|z|\to\infty ,
\]
and the set $\{a_1,\ldots,a_N\}$ is finite, we may choose $R>1$
sufficiently large so that
\[
\bigl|\psi^\mu(z)-B-z\bigr|<1
\qquad\text{for }|z|=R,
\]
and
\[
\bigl|(F^\mu)^{-1}(a_j)\bigr|<R,
\qquad j=1,\ldots,N .
\]

Parametrize the circle $|z|=R$ positively by
\[
\gamma_R(t):=Re^{2\pi it},
\qquad 0\le t\le1 .
\]
Consider the straight-line homotopy
\begin{align*}
H(s,t)
:=
(1-s)\gamma_R(t)
+s\bigl(\psi^\mu(\gamma_R(t))-B\bigr).
\end{align*}
For every $(s,t)\in[0,1]^2$,
\[
\begin{aligned}
|H(s,t)|
&=
\left|
\gamma_R(t)
+s\bigl(
\psi^\mu(\gamma_R(t))-B-\gamma_R(t)
\bigr)
\right|\\
&\ge
R
-
s\bigl|
\psi^\mu(\gamma_R(t))-B-\gamma_R(t)
\bigr|\\
&>
R-s>0 .
\end{aligned}
\]
Thus $H$ is a homotopy in $\C\setminus\{0\}$ from $\gamma_R$ to
$\psi^\mu\circ\gamma_R-B$.  By Lemma~\ref{lem-winding-homotopy} in the Appendix \ref{appendix B}, we have 
\begin{align*}
\Ind\bigl(\psi^\mu\circ\gamma_R-B,0\bigr)
=
\Ind(\gamma_R,0)
=
1 .
\end{align*}
where $\operatorname{Ind}(\gamma,z)$ denotes the winding number of $\gamma$ about $z$; see Appendix \ref{appendix B} .

Set $\widetilde\gamma_R:=F^\mu\circ\gamma_R $.
Since $F^\mu$ is an orientation-preserving homeomorphism,
$\widetilde\gamma_R$ is a positively oriented Jordan curve.  Moreover, one has $\psi^\mu\circ\gamma_R
=
\Phi\circ\widetilde\gamma_R$. Hence
\[
\Ind\bigl(\Phi\circ\widetilde\gamma_R-B,0\bigr)=1 .
\]
By the choice of $R$, every $a_j$ lies in the bounded component of
$\C\setminus\widetilde\gamma_R$.  Therefore
\[
\Ind(\widetilde\gamma_R-a_j,0)=1,
\qquad j=1,\ldots,N .
\]

For every $t\in[0,1]$,
\[
\Phi(\widetilde\gamma_R(t))-B
=
c\prod_{j=1}^{N}
\bigl(\widetilde\gamma_R(t)-a_j\bigr)^{m_j}.
\]
The constant loop $t\mapsto c$ has winding number zero.  Hence, by repeated application of Lemma~\ref{lem-winding-product} in the Appendix \ref{appendix B}, we have
\begin{align*}
\begin{aligned}
1
=
\Ind\bigl(\Phi\circ\widetilde\gamma_R-B,0\bigr)
=
\sum_{j=1}^{N}
m_j\Ind(\widetilde\gamma_R-a_j,0)
=
\sum_{j=1}^{N}m_j
=
\deg\Phi .
\end{aligned}
\end{align*}
Thus $\deg\Phi=1$, so $\Phi(w)=aw+b$ with $ a\neq0$ .

Hence $\Phi$ is a conformal automorphism of $\C$, and therefore
$\psi^\mu=\Phi\circ F^\mu$ is a quasiconformal homeomorphism of $\C$
onto itself.

Finally, suppose that $\psi_1$ and $\psi_2$ are two quasiconformal
homeomorphisms satisfying the stated conditions.  Since they have the
same Beltrami coefficient,
\[
G:=\psi_1\circ\psi_2^{-1}
\]
is a conformal automorphism of $\C$, so $G(w)=aw+b$ for some
$a\neq0$.  Since $\psi_1=G\circ\psi_2$ and
\[
\psi_j(z)=z+B_j+o(1),
\qquad |z|\to\infty ,
\]
we have
\[
z+B_1+o(1)
=
a\bigl(z+B_2+o(1)\bigr)+b .
\]
Comparing the leading terms gives $a=1$.  Since
$\psi_1(0)=\psi_2(0)=0$, we also have $b=0$.  Thus $G=\mathrm{id}$
and $\psi_1=\psi_2$.  This proves the theorem.
\end{proof}
Throughout this section we assume that the Beltrami differential $\nu$
admits the decomposition
\begin{align}\label{eqn-general-n-decay}
\nu_0(z):=\nu(z)-a\,\frac{\bar z^{\,n-1}}{z^{n-1}}\in L^r(\C)
\quad\text{for some }r\in(0,2).
\end{align}

\subsection{Proof of the existence part of Theorem~\ref{thm-main-hh}}

We first prove the existence assertion in Theorem~\ref{thm-main-hh}.

\begin{theorem}\label{thm-yibande-n-new}
Let the Beltrami differential $\nu\in L^\infty(\C)$ satisfy
$\|\nu\|_\infty\le k<1$.  Assume there exists a constant
$a\in(-1,1)$ and an integer $n\ge1$ such that the global perturbation term satisfies
\eqref{eqn-general-n-decay}.  Then there exists a unique quasiconformal
map $\phi=\phi^\nu\colon\Ch\to\Ch$ satisfying:
\begin{enumerate}
\item Beltrami equation:
$\partial_{\bar z}\phi=\nu\,\partial_z\phi$;
\item Normalization condition: $\phi(0)=0$ and the following asymptotic
expansion holds:
\begin{align*}
\phi(z)=\bigl(z^n+a\bar z^{\,n}\bigr)^{1/n}+B+o(1),
\quad\text{as }z\to\infty ,
\end{align*}
where the constant $B\in\C$ is uniquely determined by $\nu$;
\item there exist exponents $1<q<2<p<\infty$ and
$h\in L^q(\C)\cap L^p(\C)$ such that
\begin{align*}
\phi(z)
=
\bigl(z^n+a\bar z^{\,n}\bigr)^{1/n}
+
P(h)\!\left(\bigl(z^n+a\bar z^{\,n}\bigr)^{1/n}\right).
\end{align*}
\end{enumerate}
\end{theorem}

\begin{proof}
Set
\[
\mu_F(z):=a\,\frac{\bar z^{\,n-1}}{z^{n-1}},
\qquad
\nu_0:=\nu-\mu_F ,
\]
and consider the following quasiconformal map:
\begin{align*}
F(z):=\bigl(z^n+a\bar z^{\,n}\bigr)^{1/n}.
\end{align*}
More precisely, for $z\neq0$ we take
\[
F(z)
=
z\left(
1+a\Bigl(\frac{\bar z}{z}\Bigr)^{n}
\right)^{1/n},
\qquad
F(0)=0,
\]
\review{where the $n$-th root is defined using the principal logarithm on
the right half-plane, which contains $\overline{D}(1,|a|)$.
Thus the definition is valid for every $|a|<1$, including $a=0$.}
Moreover,
\[
(1-|a|)^{1/n}|z|
\le |F(z)|
\le(1+|a|)^{1/n}|z|,
\]
so that $|F(z)|\asymp|z|$.
It is easy to compute the Jacobian of $F$:
\[
J_F(z)
=
(1-a^2)
\left(
\frac{|z|}{|F(z)|}
\right)^{2(n-1)}\le C
\]
where \(C>0\) depends only on \(a\) and \(n\).
\begin{reviewchange}
We verify the global homeomorphism property before using $F^{-1}$.
In polar coordinates write
\[
 F(re^{i\theta})=rR(\theta)e^{i\Theta(\theta)},\qquad
 R(\theta)=|1+ae^{-2in\theta}|^{1/n},\qquad
 \Theta(\theta)=\theta+\frac1n\operatorname{Im}\log(1+ae^{-2in\theta}).
\]
Then
\[
 \Theta'(\theta)=\frac{1-a^2}{|1+ae^{-2in\theta}|^2}>0,
 \qquad \Theta(\theta+2\pi)=\Theta(\theta)+2\pi.
\]
Since $R$ is bounded above and below by positive constants, this
defines a homeomorphism of $\Ch$ fixing $0$ and $\infty$.
For $z\ne0$,
\[
 F_z=z^{n-1}F^{1-n},\qquad
 F_{\bar z}=a\bar z^{\,n-1}F^{1-n},\qquad
 F_{\bar z}/F_z=a(\bar z/z)^{n-1}.
\]
These derivatives are bounded; integrating along line segments,
splitting at $0$ when needed, shows that $F$ is locally Lipschitz
across $0$. It is therefore quasiconformal with coefficient $\mu_F$.
The displayed bounds for $|F(z)|/|z|$ also give
$0<c\le J_F(z)\le C<\infty$ almost everywhere.

\end{reviewchange}
Define the following Beltrami differential
\begin{align*}
\lambda
=
\left[
\frac{\nu-\mu_F}
{1-\nu\overline{\mu_F}}
\frac{\partial_zF}
{\overline{\partial_zF}}
\right]\circ F^{-1}.
\end{align*}
\noindent\textbf{Claim}:   we have $\lambda\in L^r(\C)$ and $\|\lambda\|_\infty<1$.

Indeed, notice that
\[
|\lambda(F(z))|
=
\left|
\frac{\nu(z)-\mu_F(z)}
{1-\nu(z)\overline{\mu_F(z)}}
\right|.
\]
The right-hand side is the pseudohyperbolic distance between $\nu(z)$
and $\mu_F(z)$.  Since $|\nu|\le k<1$ and $|\mu_F|=|a|<1$, there
exists $\kappa<1$, depending only on $k$ and $|a|$, such that
\[
\|\lambda\|_\infty\le\kappa<1 .
\]
Moreover,
\[
|\lambda(F(z))|
\le
\frac{1}{1-k|a|}\,|\nu(z)-\mu_F(z)|
=
\frac{1}{1-k|a|}\,|\nu_0(z)|.
\]
Hence, using the change of variables $w=F(z)$ and the uniform bounds
for $J_F$,
\begin{align*}
\int_{\C}|\lambda(w)|^r\,dA(w)
=
\int_{\C}
|\lambda(F(z))|^rJ_F(z)\,dA(z)
\le
C\int_{\C}|\nu_0(z)|^r\,dA(z)<\infty .
\end{align*}
Thus $\lambda\in L^r(\C)$ and $\|\lambda\|_\infty<1$. This completes the proof of the \textbf{Claim}.

Now, by the \textbf{Claim} above and Theorem~\ref{thm-QC-new-rd-1-modified}, there exists a unique
quasiconformal homeomorphism $\Psi\colon\C\to\C$
satisfying
\[
\partial_{\bar w}\Psi
=
\lambda\,\partial_w\Psi,
\qquad
\Psi(0)=0,
\]
and there exist  $1<q<2<p<\infty$ such that
\[
h:=(I-\lambda T)^{-1}\lambda
\in L^q(\C)\cap L^p(\C)
\]
and
\[
\Psi(w)
=
w+P(h)(w)
=
w+C(h)(w)-C(h)(0).
\]
In particular,
\[
\Psi(w)=w+B+o(1),
\qquad w\to\infty ,
\]
where
\[
B=-C(h)(0)
=
\frac1\pi
\iint_{\C}
\frac{h(\zeta)}{\zeta}\dA (\zeta).
\]
%
%
%
%
Define the quasiconformal homeomorphism as follows:
\[
\phi:=\Psi\circ F .
\]
Then $\mu_\phi=\nu$,
and hence
\[
\partial_{\bar z}\phi
=
\nu\,\partial_z\phi
\qquad\text{a.e.\ on }\C .
\]
Also,
\[
\phi(0)=\Psi(F(0))=\Psi(0)=0 .
\]
Furthermore, 
\[
\phi(z)
=
F(z)+P(h)(F(z)),
\]
that is,
\[
\phi(z)
=
\bigl(z^n+a\bar z^{\,n}\bigr)^{1/n}
+
P(h)\!\left(\bigl(z^n+a\bar z^{\,n}\bigr)^{1/n}\right),
\]
which proves assertion~(3).  Since $|F(z)|\asymp|z|$, we have
$F(z)\to\infty$ as $z\to\infty$, and therefore
\[
\phi(z)
=
\Psi(F(z))
=
F(z)+B+o(1)
=
\bigl(z^n+a\bar z^{\,n}\bigr)^{1/n}+B+o(1).
\]
This proves the required asymptotic normalization.

It remains to prove uniqueness.  Suppose that $\phi_1$ and $\phi_2$
both satisfy the conclusions of the theorem, and set
\[
\Psi_j:=\phi_j\circ F^{-1},
\qquad j=1,2 .
\]
By the Beltrami composition formula, both $\Psi_1$ and $\Psi_2$ have
Beltrami coefficient $\lambda$.  Moreover,
\[
\Psi_j(0)=0,
\]
and, since $F^{-1}(w)\to\infty$ as $w\to\infty$,
\[
\Psi_j(w)
=
w+B_j+o(1),
\qquad w\to\infty .
\]
The uniqueness part of Theorem~\ref{thm-QC-new-rd-1-modified}
therefore yields $\Psi_1=\Psi_2$, and consequently
\[
\phi_1=\Psi_1\circ F=\Psi_2\circ F=\phi_2 .
\]
In particular, $B_1=B_2$, so the constant $B$ is uniquely determined
by $\nu$.
\end{proof}

\begin{corollary}\label{coro-hyl-67-modified-new}
Let the Beltrami differential $\nu\in L^\infty(\C)$ satisfy
$\|\nu\|_\infty\le k<1$.  Assume there exists a constant
$a\in(-1,1)$ and an integer $n\ge1$ such that
$\nu_0(z)=\nu(z)-a\bar z^{\,n-1}/z^{n-1}\in L^r(\C)$ for some
$r\in(0,2)$. For $\eps\in\R$ with $|\eps|$ sufficiently small, let $\phi^{\eps\nu}$
be the quasiconformal map corresponding to the coefficient $\eps\nu$
 obtained in Theorem~\ref{thm-yibande-n-new}.  Then the following
first-order variational expansion holds uniformly on $\C$:
\begin{align*}
\phi^{\eps\nu}(z)
=
\bigl(z^n+\eps a\bar z^{\,n}\bigr)^{1/n}
+\eps P(\nu_0)(z)+O(\eps^2).
\end{align*}
In particular, as $z\to\infty$ the asymptotic behavior reduces to
\begin{align*}
\phi^{\eps\nu}(z)
=
\bigl(z^n+\eps a\bar z^{\,n}\bigr)^{1/n}
+\frac{\eps}{\pi}
\iint_{\C}\frac{\nu_0(\zeta)}{\zeta}\dA (\zeta)
+\eps\,o(1)+O(\eps^2),
\end{align*}
where the $O(\eps^2)$ term is uniform in $z\in\C$, while the
$o(1)$ term is taken as $z\to\infty$ for each fixed sufficiently
small $|\eps|$.
\end{corollary}

\begin{proof}
We again use the decomposition strategy employed earlier.  First
define
\begin{align}\label{eqn-A}
A_\eps(z)=\bigl(z^n+\eps a\bar z^{\,n}\bigr)^{1/n}.
\end{align}
For sufficiently small $|\eps|$, $A_\eps$ is a homeomorphism of $\C$
onto itself.  Let $\mu_0(z)=a\bar z^{\,n-1}/z^{n-1}$.  A direct
computation shows that the Beltrami coefficient of $A_\eps$ is exactly
$\eps\mu_0(z)$.

Decompose the map $\phi^{\eps\nu}$ as
\begin{align}\label{eqn-linshi-4-6}
\phi^{\eps\nu}(z)=\Psi_\eps(A_\eps(z)),
\end{align}
where $\Psi_\eps(w)=\phi^{\eps\nu}(A_\eps^{-1}(w))$ is also a
quasiconformal map satisfying the normalization at infinity:
\[
\Psi_\eps(0)=0,
\qquad
\Psi_\eps(w)=w+B_\eps+o(1), \quad \text{as }w\to\infty.
\]
By the composition formula for Beltrami differentials, the Beltrami
coefficient $\lambda_\eps$ of $\Psi_\eps$ at the point
$w=A_\eps(z)$ is given by
\begin{align*}
\lambda_\eps(A_\eps(z))
=
\frac{\eps\nu(z)-\eps\mu_0(z)}
{1-\overline{\eps\mu_0(z)}\,\eps\nu(z)}
\frac{\partial_zA_\eps}{\overline{\partial_zA_\eps}}
=
\frac{\eps\nu_0(z)}
{1-\eps^2\overline{\mu_0(z)}\nu(z)}
\frac{\partial_zA_\eps}{\overline{\partial_zA_\eps}} .
\end{align*}

Notice that
$\partial_zA_\eps(z)=z^{n-1}/A_\eps(z)^{n-1}$.  Since
$A_\eps(z)=z(1+\eps a(\bar z/z)^n)^{1/n}=z(1+O(|\eps|))$, we have
uniformly
$\frac{\partial_zA_\eps}{\overline{\partial_zA_\eps}}=1+O(|\eps|)$.
Because $\mu_0,\nu\in L^\infty(\C)$, for sufficiently small $|\eps|$ we
obtain the uniform asymptotic expansion
\begin{align}\label{eqn-linshi-7}
\lambda_\eps(A_\eps(z))
=
\eps\nu_0(z)+O(\eps^2|\nu_0(z)|)=O(|\eps\nu_0(z)|).
\end{align}
Then
\[
\|\lambda_\eps\|_\infty=O(|\eps|).
\]

To evaluate the $L^p$ norm of $\lambda_\eps$ with respect to the area
measure $dw$, we apply the change of variables $w=A_\eps(z)$.  The
Jacobian is uniformly bounded as
$J_{A_\eps}(z)=|\partial_zA_\eps|^2-|\partial_{\bar z}A_\eps|^2
=1+O(|\eps|)$.  Since $\nu_0\in L^r(\C)$ for some $r\in(0,2)$,
integrating the pointwise bound \eqref{eqn-linshi-7} yields
\begin{align}\label{eqn-linshi}
\|\lambda_\eps\|_r^r
=
\iint_{\C}
|\lambda_\eps(A_\eps(z))|^rJ_{A_\eps}(z)\dA (z)
=
O(|\eps|^r),
\end{align}
which rigorously implies $\lambda_\eps\in L^r(\C)$ and
$\|\lambda_\eps\|_r=O(|\eps|)$. Since $\lambda_\eps\in L^r(\C)\cap L^\infty(\C)$, the estimate $\|\lambda_\eps\|_s^s\le \|\lambda_\eps\|_\infty^{\,s-r}\|\lambda_\eps\|_r^r$ for $s\ge r$ yields $\|\lambda_\eps\|_s=O(|\eps|)$.
Choose $p>2$ sufficiently close to $2$ so that for small $\eps$ we have
$\|\lambda_\eps\|_\infty\|T\|_p<1$.  Then there exists
$h_\eps=(I-\lambda_\eps T)^{-1}(\lambda_\eps)\in L^p(\C)$ such that
\begin{align}\label{eqn-linshi-4}
\Psi_\eps(w)=w+P(h_\eps)(w).
\end{align}
Consequently,
\begin{align}\label{eqn-linshi-2}
\|h_\eps-\lambda_\eps\|_p
=
\|\lambda_\eps T(h_\eps)\|_p
\le
\|\lambda_\eps\|_\infty\|T\|_p\|h_\eps\|_p .
\end{align}
Since $\|\lambda_\eps\|_\infty=O(|\eps|)$, we have
\[
\|\lambda_\eps T\|_p
\le
\|\lambda_\eps\|_\infty\|T\|_p
=
O(|\eps|).
\]
Together with \eqref{eqn-linshi} this yields
\[
\|h_\eps\|_p
\le
\|(I-\lambda_\eps T)^{-1}\|_p\|\lambda_\eps\|_p
\le
\frac{1}{1-\|\lambda_\eps T\|_p}\|\lambda_\eps\|_p
=
O(|\eps|).
\]
Returning to \eqref{eqn-linshi-2} we obtain
\begin{align*}
\|\lambda_\eps T(h_\eps)\|_p
=
\|h_\eps-\lambda_\eps\|_p
=
O(\eps^2).
\end{align*}
\review{The same argument can be carried out for an exponent $\max\{1,r\}<q<2$ sufficiently close to $2$}, and the Neumann series construction gives
the same $h_\eps$.  That is, there exists $1<q<2$ such that
\begin{align*}
\|\lambda_\eps T(h_\eps)\|_q
=
\|h_\eps-\lambda_\eps\|_q
=
O(\eps^2).
\end{align*}
Hence
\begin{align*}
\begin{split}
\|P(h_\eps)-P(\lambda_\eps)\|_\infty
&\le
2\|C(h_\eps)-C(\lambda_\eps)\|_\infty\\
&\le
2M_{p,q}\bigl(
\|h_\eps-\lambda_\eps\|_p+\|h_\eps-\lambda_\eps\|_q
\bigr)
=
O(\eps^2).
\end{split}
\end{align*}
Returning to \eqref{eqn-linshi-4} we get
\begin{align}\label{eqn-linshi-45}
\Psi_\eps(w)
=
w+P(h_\eps)(w)
=
w+P(\lambda_\eps)(w)+R(\eps,w),
\end{align}
where the remainder satisfies
$\sup_{w\in\C}|R(\eps,w)|=O(\eps^2)$.\footnote{Since $C(h_\eps)$ and
$C(\lambda_\eps)$ are uniformly bounded on $\C$ via
Lemma~\ref{thm-astal-modify}, the error term $R(\eps,w)$ in
\eqref{eqn-linshi-45} is strictly bounded in $L^\infty(\C)$
independent of $w$.}

Next we change variables $\zeta=A_\eps(\eta)$ to compute
$P(\lambda_\eps)$ at $w=A_\eps(z)$.  First note that the Jacobian of
the change of variables is
\[
J_{A_\eps}(\eta)
=
|\partial_\eta A_\eps|^2-|\partial_{\bar\eta}A_\eps|^2
=
1+O(|\eps|),
\quad\text{uniformly.}
\]
We have
\begin{align*}
P(\lambda_\eps)(A_\eps(z))
&=
-\frac{1}{\pi}\iint_{\C}
\lambda_\eps(\zeta)
\left(
\frac{1}{\zeta-A_\eps(z)}-\frac{1}{\zeta}
\right)\dA (\zeta)
\\
&=
-\frac{1}{\pi}\iint_{\C}
\lambda_\eps(A_\eps(\eta))
\left(
\frac{1}{A_\eps(\eta)-A_\eps(z)}-\frac{1}{A_\eps(\eta)}
\right)
J_{A_\eps}(\eta)\dA (\eta).
\end{align*}
Recall from \eqref{eqn-linshi-7} that
$\lambda_\eps(A_\eps(\eta))J_{A_\eps}(\eta)
=\eps\nu_0(\eta)+O(\eps^2|\nu_0(\eta)|)$.  Substituting this into the
integral yields
\begin{align*}
P(\lambda_\eps)(A_\eps(z))
&=
-\frac{\eps}{\pi}\iint_{\C}
\nu_0(\eta)
\left(
\frac{1}{A_\eps(\eta)-A_\eps(z)}-\frac{1}{A_\eps(\eta)}
\right)\dA (\eta)
\\
&\quad
-\frac{1}{\pi}\iint_{\C}
O(\eps^2|\nu_0(\eta)|)
\left(
\frac{1}{A_\eps(\eta)-A_\eps(z)}-\frac{1}{A_\eps(\eta)}
\right)\dA (\eta).
\end{align*}
By Lemma~\ref{lemma-A-epsilon-difference}, which will be proved below,
for sufficiently small $|\eps|$ there exists a constant $c>0$ such that
$|A_\eps(\eta)-A_\eps(z)|\ge c|\eta-z|$ uniformly, and
$A_\eps(\eta)-A_\eps(z)=(\eta-z)+O(|\eps||\eta-z|)$.  Therefore,
\begin{align*}
\left|
\frac{1}{A_\eps(\eta)-A_\eps(z)}-\frac{1}{\eta-z}
\right|
\le
\frac{C|\eps|}{|\eta-z|}.
\end{align*}
Similarly,
\review{$\left|\frac{1}{A_\eps(\eta)}-\frac{1}{\eta}\right|\le\frac{C|\eps|}{|\eta|}$}.  To estimate both the kernel-replacement
term and the coefficient-remainder term, observe that the full
perturbed kernel satisfies
\begin{align*}
\left|
\frac{1}{A_\eps(\eta)-A_\eps(z)}
-\frac{1}{A_\eps(\eta)}
\right|
\le
C\left(\frac1{|\eta-z|}+\frac1{|\eta|}\right).
\end{align*}
Because $\nu_0\in L^r\cap L^\infty$, we may choose
$\max\{1,r\}<q<2<p<\infty$
such that $\nu_0\in L^q\cap L^p$.
Writing
\[
K_\eps(\eta,z)=\frac1{A_\eps(\eta)-A_\eps(z)}-\frac1{A_\eps(\eta)},
\qquad
K_0(\eta,z)=\frac1{\eta-z}-\frac1\eta,
\]
the bounds above read
\[
|K_\eps|\le C\bigl(|\eta-z|^{-1}+|\eta|^{-1}\bigr),
\qquad
|K_\eps-K_0|\le C|\eps|\bigl(|\eta-z|^{-1}+|\eta|^{-1}\bigr),
\]
while
$\lambda_\eps(A_\eps(\eta))J_{A_\eps}(\eta)
=\eps\nu_0(\eta)+e_\eps(\eta)$ with
$|e_\eps(\eta)|\le C\eps^2|\nu_0(\eta)|$.  The absolute-potential
estimate \eqref{eqn-cauchy-absolute-potential}, applied to
$h=|\nu_0|$ once centered at $z$ and once at $0$, therefore gives
\[
\begin{split}
\bigl|P(\lambda_\eps)(A_\eps(z))-\eps P(\nu_0)(z)\bigr|
&\le
\frac1\pi\iint_{\C}
\Bigl(\review{|\eps|}\,|\nu_0(\eta)|\,|K_\eps-K_0|
+|e_\eps(\eta)|\,|K_\eps|\Bigr)\dA (\eta)\\
&=
O(\eps^2),
\end{split}
\]
uniformly for all $z\in\C$, each term being bounded by a constant
multiple of $\eps^2\bigl(\|\nu_0\|_p+\|\nu_0\|_q\bigr)$. Thus
\begin{align}\label{eqn-linshi-4-5}
P(\lambda_\eps)(A_\eps(z))
=
\eps P(\nu_0)(z)+O(\eps^2),
\end{align}
and the $O(\eps^2)$ term is uniform in $z\in\C$.

Combining \eqref{eqn-linshi-4-6}, \eqref{eqn-linshi-45} and
\eqref{eqn-linshi-4-5} we have
\begin{align}
\phi^{\eps\nu}(z)
&=
\Psi_\eps(A_\eps(z))
=
A_\eps(z)+P(\lambda_\eps)(A_\eps(z))+R(\eps,A_\eps(z))
\notag\\
&=
\bigl(z^n+\eps a\bar z^{\,n}\bigr)^{1/n}
+\eps P(\nu_0)(z)+O(\eps^2).
\label{eqn-linshi-4-7}
\end{align}
Finally, the same $L^q\cap L^p$ estimate gives $C(\nu_0)(z)\to0$ at
infinity, and hence
\[
P(\nu_0)(z)
=
\frac{1}{\pi}\iint_{\C}\frac{\nu_0(\zeta)}{\zeta}\dA (\zeta)+o(1).
\]
Substituting this into the expression for $\phi^{\eps\nu}(z)$ in
\eqref{eqn-linshi-4-7} gives
\[
\phi^{\eps\nu}(z)
=
\bigl(z^n+\eps a\bar z^{\,n}\bigr)^{1/n}
+\frac{\eps}{\pi}
\iint_{\C}\frac{\nu_0(\zeta)}{\zeta}\dA (\zeta)
+\eps\,o(1)+O(\eps^2).
\]
This completes the proof.
\end{proof}

\begin{remark}\label{remark-half-plane-expansion}
Suppose that the Beltrami differential $\nu$ is originally defined on
$\HH$ and is extended to $\C$ by the symmetric reflection
\[
\nu(z)=\overline{\nu(\bar z)},
\qquad z\in\C\setminus\overline{\HH}.
\]
Since $a\in(-1,1)$ is real, the same symmetry is satisfied by
\[
\nu_0(z)
=
\nu(z)-a\,\frac{\bar z^{\,n-1}}{z^{n-1}},
\]
namely $\nu_0(z)=\overline{\nu_0(\bar z)}$.
\begin{reviewchange}
The corresponding quasiconformal map $\phi^\nu$ preserves $\HH$.
Indeed, $F(\bar z)=\overline{F(z)}$, and
$z\mapsto\overline{\phi^\nu(\bar z)}$ has the same coefficient and
normalization as $\phi^\nu$. Uniqueness in
Theorem~\ref{thm-yibande-n-new} gives reflection symmetry and $B\in\R$.
On the positive imaginary axis, $F(iy)$ is a positive imaginary
multiple of $y$, and $\phi^\nu(iy)=F(iy)+B+o(1)$ as $y\to\infty$.
Thus the upper half-plane is mapped to itself.
\end{reviewchange}
\end{remark}

\begin{lemma}\label{lemma-A-epsilon-difference}
Let $A_\eps(z)$ be defined by \eqref{eqn-A}, where $n\ge1$ and
$a\in\R$.  For $z\neq0$, we take the branch
\[
A_\eps(z)
=
z\left(
1+\eps a\Bigl(\frac{\bar z}{z}\Bigr)^{n}
\right)^{1/n},
\]
and set $A_\eps(0)=0$.  Then there exist $\eps_0>0$ and $C>0$,
depending only on $a$ and $n$, such that for every
$|\eps|<\eps_0$ and all $z,\eta\in\C$,
\begin{align*}
\left|
A_\eps(\eta)-A_\eps(z)-(\eta-z)
\right|
\le
C|\eps|\,|\eta-z|.
\end{align*}
Equivalently,
\begin{align}\label{eqn-A-epsilon-difference}
A_\eps(\eta)-A_\eps(z)
=
(\eta-z)+O\bigl(|\eps|\,|\eta-z|\bigr),
\end{align}
where the $O$-term is uniform with respect to $z,\eta\in\C$.

In particular, after decreasing $\eps_0$ if necessary,
\[
(1-C|\eps|)|\eta-z|
\le
|A_\eps(\eta)-A_\eps(z)|
\le
(1+C|\eps|)|\eta-z|,
\]
and hence $A_\eps$ is uniformly bi-Lipschitz for sufficiently small
$|\eps|$.
\end{lemma}
\begin{proof}
For $z\neq0$, differentiating
\[
A_\eps(z)^n=z^n+\eps a\bar z^{\,n}
\]
gives
\[
\partial_zA_\eps(z)
=
z^{n-1}A_\eps(z)^{1-n},
\qquad
\partial_{\bar z}A_\eps(z)
=
\eps a\bar z^{\,n-1}A_\eps(z)^{1-n}.
\]
Since
\[
A_\eps(z)
=
z\left(
1+\eps a\Bigl(\frac{\bar z}{z}\Bigr)^n
\right)^{1/n},
\qquad
\left|\Bigl(\frac{\bar z}{z}\Bigr)^n\right|=1,
\]
the choice of branch in the definition of $A_\eps$ and sufficiently
small $|\eps|$ imply, uniformly in $z\neq0$,
\[
\partial_zA_\eps(z)=1+O(|\eps|),
\qquad
\partial_{\bar z}A_\eps(z)=O(|\eps|).
\]
Hence, setting $F_\eps:=A_\eps-\mathrm{id}$,
\[
\|DF_\eps(z)\|_{\mathrm{op}}
\le
|\partial_zF_\eps(z)|+|\partial_{\bar z}F_\eps(z)|
\le C|\eps|,
\qquad z\neq0,
\]
where $\|\cdot\|_{\mathrm{op}}$ denotes the operator norm of the real
differential.  Integrating along the line segment joining $z$ and
$\eta$ gives
\[
|F_\eps(\eta)-F_\eps(z)|
\le C|\eps|\,|\eta-z|
\]
whenever the segment does not pass through $0$.  If it does pass
through $0$, we split the segment at $0$ and use $F_\eps(0)=0$ and
$|z|+|\eta|=|\eta-z|$ to obtain the same estimate.  Therefore
\[
\left|
A_\eps(\eta)-A_\eps(z)-(\eta-z)
\right|
\le
C|\eps|\,|\eta-z|.
\]
This proves \eqref{eqn-A-epsilon-difference}.  The two-sided estimate
follows immediately from the triangle inequality.
\end{proof}

%
%

\subsection{Proof of the variational formulas in
Theorem~\ref{thm-main-hh}}
\label{sec-normalization-transfer}
We now return to the differential
\begin{align}\label{eqn-general-n-transfer}
\nu(z)
=
a\Bigl(\frac{\bar z}{z}\Bigr)^{n-1}+\nu_0(z),
\qquad
\nu_0\in L^r(\C),\quad 0<r<2,
\end{align}
where $\nu\in L^\infty(\C)$ is the symmetric extension from $\HH$ to
$\C$.  Let $\phi^{\eps\nu}$ be the asymptotically normalized map
constructed in Corollary~\ref{coro-hyl-67-modified-new}, and let
$\psi_{\mathrm{std}}^{\eps\nu}$ be the standard solution fixing
$0,1$ and $\infty$.

\begin{lemma}\label{lemma-normalization-transfer}
Set
\begin{align*}
\rho_\eps:=\phi^{\eps\nu}(1).
\end{align*}
For \(\eps \in \mathbb{R}\) with sufficiently small $|\eps|$, $\rho_\eps>0$ and
\begin{align}\label{eqn-normalization-transfer}
\psi_{\mathrm{std}}^{\eps\nu}(z)
=
\rho_\eps^{-1}\phi^{\eps\nu}(z).
\end{align}
Moreover, $\rho_0=1$ and 
\begin{align}
\dot\rho_0
:=
\left.\frac{\dd\rho_\eps}{\dd\eps}\right|_{\eps=0}
=
-\frac{1}{\pi}\PV
\iint_{\C}
\frac{\nu(z)}{z(z-1)}\dA (z)
=
-\frac{2}{\pi}\PV\operatorname{Re}
\iint_{\HH}
\frac{\nu(z)}{z(z-1)}\dA (z).
\label{eqn-rho-variation-half-plane}
\end{align}
Here the principal value is defined by circular truncation
at infinity.
\end{lemma}

\begin{proof}
The maps $\phi^{\eps\nu}$ and
$\psi_{\mathrm{std}}^{\eps\nu}$ have the same Beltrami
coefficient, preserve $\HH$, and fix $0$ and $\infty$.
Therefore,
$\phi^{\eps\nu}\circ
(\psi_{\mathrm{std}}^{\eps\nu})^{-1}$
is a conformal automorphism of $\HH$ fixing $0$ and $\infty$,
and hence equals $w\mapsto cw$ for some $c>0$.
Thus
\[
\phi^{\eps\nu}(z)
=c\,\psi_{\mathrm{std}}^{\eps\nu}(z).
\]
Evaluating at $z=1$ and using
$\psi_{\mathrm{std}}^{\eps\nu}(1)=1$ gives
$c=\phi^{\eps\nu}(1)=\rho_\eps>0$, proving
\eqref{eqn-normalization-transfer}.

By Corollary~\ref{coro-hyl-67-modified-new},
\[
\rho_\eps
=
(1+\eps a)^{1/n}
+\eps P(\nu_0)(1)
+O(\eps^2),
\]
and therefore
\begin{align}\label{eqn-rho-before-model}
\dot\rho_0
=
\frac{a}{n}+P(\nu_0)(1)
=
\frac{a}{n}
-\frac{1}{\pi}
\iint_{\C}
\frac{\nu_0(z)}{z(z-1)}\dA (z).
\end{align}
The integral involving $\nu_0$ is absolutely convergent.  
%
It remains to evaluate the contribution of the model term.  Put
\[
m_n(z):=\left(\frac{\bar z}{z}\right)^{n-1}.
\]
For $R>1$, writing $z=re^{i\theta}$ gives
\[
\iint_{|z|<R}
\frac{m_n(z)}{z(z-1)}\dA(z)
=
\int_0^R\int_0^{2\pi}
\frac{e^{-i(2n-1)\theta}}
{re^{i\theta}-1}
\,d\theta\,dr.
\]
For $0<r<1$, we have
\[
\frac{1}{re^{i\theta}-1}
=
-\sum_{k=0}^{\infty}r^ke^{ik\theta},
\]
and hence
\[
\int_0^{2\pi}
\frac{e^{-i(2n-1)\theta}}
{re^{i\theta}-1}\,d\theta
=
-2\pi r^{2n-1}.
\]
For $r>1$,
\[
\frac{1}{re^{i\theta}-1}
=
\sum_{k=0}^{\infty}
r^{-k-1}e^{-i(k+1)\theta},
\]
so that
\[
\int_0^{2\pi}
\frac{e^{-i(2n-1)\theta}}
{re^{i\theta}-1}\,d\theta
=
0.
\]
Therefore, for every $R>1$,
\[
\iint_{|z|<R}
\frac{m_n(z)}{z(z-1)}\dA(z)
=
-2\pi\int_0^1 r^{2n-1}\,dr
=
-\frac{\pi}{n}.
\]
Letting $R\to\infty$ and using the circular truncation defining the
principal value, we obtain
\begin{align}\label{eqn-model-normalization-integral}
\PV
\iint_{\C}
\frac{m_n(z)}{z(z-1)}\dA(z)
=
-\frac{\pi}{n}.
\end{align}
Substitution of \eqref{eqn-model-normalization-integral} into
\eqref{eqn-rho-before-model} gives the first equality in
\eqref{eqn-rho-variation-half-plane}.  The second equality in
\eqref{eqn-rho-variation-half-plane} follows from symmetry.
\end{proof}

We shall  use  the notation introduced in \S~\ref{sec-main-sdh}, recall that
under a positive dilation $K\mapsto cK$, the driving function is
multiplied by $c$ and the half-plane capacity by $c^2$.  Equation
\eqref{eqn-normalization-transfer} therefore yields
\begin{align}\label{eqn-driving-scaling-transfer}
\lambda_{t,\mathrm{asy}}^{\eps\nu}
=
\rho_\eps\,
\lambda_{t,\mathrm{std}}^{\eps\nu},\qquad
a_{t,\mathrm{asy}}^{\eps\nu}
=
\rho_\eps^2\,
a_{t,\mathrm{std}}^{\eps\nu}.
\end{align}

\begin{theorem}\label{thm-universal-n}
Under the hypotheses of Theorem~\ref{thm-main-hh}, with $\nu$
extended symmetrically to $\C$ as in \eqref{eqn-general-n-transfer},
the following identities hold for every fixed $0\le t<T_+$:
\begin{align}
\dot\lambda_{t,\mathrm{asy}}
&=
-\frac{2}{\pi}\PV\operatorname{Re}
\iint_{D_t}
\nu(z)\left(
\frac{g_t'(z)^2}{g_t(z)-\lambda_t}-\frac1z
\right)\dA (z)
+
\frac{2}{\pi}\operatorname{Re}
\iint_{K_t^\circ}\frac{\nu(z)}{z}\dA (z),
\label{eqn-uni-drive}\\
\dot a_{t,\mathrm{asy}}
&=
\frac{1}{\pi}\PV\operatorname{Re}
\iint_{D_t}
\nu(z)\bigl(g_t'(z)^2-1\bigr)\dA (z)
-
\frac{1}{\pi}\operatorname{Re}
\iint_{K_t^\circ}\nu(z)\dA (z).
\label{eqn-uni-cap}
\end{align}
Here the principal values are circular principal values at infinity: $D_t$
is replaced by $D_t\cap\{|z|<R\}$ and $R\to\infty$. 
The two trace integrals are absolutely convergent.
\end{theorem}

\begin{corollary}\label{coro-oscillatory-zero-trace}
Under the hypotheses of Theorem~\ref{thm-universal-n}, 
if $\nu=0$ almost everywhere on $K_t^\circ$, or if
$\Area(K_t^\circ)=0$, then
\begin{align*}
\dot\lambda_{t,\mathrm{asy}}
&=
-\frac{2}{\pi}\PV\operatorname{Re}
\iint_{D_t}
\nu(z)\left(
\frac{g_t'(z)^2}{g_t(z)-\lambda_t}-\frac1z
\right)\dA (z),
\\
\dot a_{t,\mathrm{asy}}
&=
\frac{1}{\pi}\PV\operatorname{Re}
\iint_{D_t}
\nu(z)\bigl(g_t'(z)^2-1\bigr)\dA (z).
\end{align*}
\begin{reviewchange}
In particular, these formulas apply to the global zero extension
in \eqref{eqn-global-zero-extension}, provided that the extended
coefficient still satisfies the oscillatory $L^r$ hypothesis of
Theorem~\ref{thm-universal-n}.
If a bounded coefficient already satisfies this hypothesis,
any bounded modification supported on a finite-area set preserves
this $L^r$ condition: the difference of the two coefficients is bounded and supported
on that set, and therefore belongs to $L^r$. In particular, setting
such a coefficient to zero on a finite-area trace is allowed.
A zero extension on a full trace of infinite area need not preserve
the hypothesis. For example, if the coefficient is
$a(\bar z/z)^{n-1}$ off the full trace and zero on it, then for
$a\ne0$ its perturbation from that model belongs to $L^r$ precisely
when the trace has finite area.
\end{reviewchange}
\end{corollary}

\begin{proof}[Proof of Theorem~\ref{thm-universal-n}]
The scaling identities
\eqref{eqn-driving-scaling-transfer}
follow directly from the dilation covariance of the hydrodynamically
normalized mapping-out maps and their boundary values at the tip.

At $\eps=0$,
\[
\rho_0=1,\qquad
\lambda_{t,\mathrm{std}}^{0}=\lambda_t,\qquad
a_{t,\mathrm{std}}^{0}=t .
\]
Differentiation gives
\begin{align}\label{eqn-driving-derivative-transfer}
\dot\lambda_{t,\mathrm{asy}}
=
\dot\lambda_{t,\mathrm{std}}+\lambda_t\dot\rho_0 ,\qquad
\dot a_{t,\mathrm{asy}}
=
\dot a_{t,\mathrm{std}}+2t\dot\rho_0 .
\end{align}
Write $R(z)=1/[z(z-1)]$.  Theorem~\ref{thm-main-normal} gives
\begin{align*}
\dot\lambda_{t,\mathrm{std}}
=
-\frac{2}{\pi}\operatorname{Re}
\iint_{D_t}
\nu(z)\left(
A_t(z)-\frac1z-\lambda_tR(z)
\right)\dA (z)
+
\frac{2}{\pi}\operatorname{Re}
\iint_{K_t^\circ}
\nu(z)\left(\frac1z+\lambda_tR(z)\right)\dA (z).
\end{align*}
On the other hand, Lemma~\ref{lemma-normalization-transfer} gives
\begin{align*}
\lambda_t\dot\rho_0
=
-\frac{2\lambda_t}{\pi}\PV\operatorname{Re}
\left(
\iint_{D_t}\nu(z)R(z)\dA (z)
+
\iint_{K_t^\circ}\nu(z)R(z)\dA (z)
\right).
\end{align*}
The $R$-terms cancel separately on $D_t$ and $K_t^\circ$.
Substituting in \eqref{eqn-driving-derivative-transfer} gives
\eqref{eqn-uni-drive}.

Similarly, Theorem~\ref{thm-main-normal} gives
\begin{align*}
\dot a_{t,\mathrm{std}}
=
\frac1\pi\operatorname{Re}
\iint_{D_t}
\nu(z)\bigl(g_t'(z)^2-1+4tR(z)\bigr)\dA (z)
+
\frac1\pi\operatorname{Re}
\iint_{K_t^\circ}
\nu(z)\bigl(-1+4tR(z)\bigr)\dA (z),
\end{align*}
whereas
\begin{align*}
2t\dot\rho_0
=
-\frac{4t}{\pi}\PV\operatorname{Re}
\left(
\iint_{D_t}\nu(z)R(z)\dA (z)
+
\iint_{K_t^\circ}\nu(z)R(z)\dA (z)
\right).
\end{align*}
Again the $R$-terms cancel on both pieces, yielding
\eqref{eqn-uni-cap}.
\begin{reviewchange}
These cancellations may be performed first on
$\HH\cap\{|z|<R\}$, with $R$ large enough to contain $K_t$.
To justify the limiting integrals directly, the two asymptotic
domain kernels have leading terms $\lambda_t/z^2$ and $-4t/z^2$,
respectively, and errors $O(|z|^{-3})$.
The product of either leading term with $a(\bar z/z)^{n-1}$ has
zero integral on each half-annulus because
$\int_0^\pi e^{-2in\theta}\,d\theta=0$.
For the perturbation, choose $\max\{1,r\}<q<2$;
$\nu_0\in L^q$ and $|z|^{-2}\in L^{q'}(\{|z|>1\})$, so the
tails involving $\nu_0$ converge absolutely. Local integrability
follows from Lemma~\ref{lemma-standard-kernel-integrability}.
Thus both circular principal values exist.
\end{reviewchange}
\end{proof}

\begin{proof}[Proof of Corollary~\ref{coro-oscillatory-zero-trace}]
Under either hypothesis, both trace integrals in
\eqref{eqn-uni-drive}--\eqref{eqn-uni-cap} vanish.
\end{proof}

\appendix

\section{Proofs of the explicit variational formulas}
\label{sec-appendix-contours}

Both Lemma~\ref{lemma-QC-deform-1} and Lemma~\ref{lemma-QC-deform-2}
rely on evaluating complex area integrals over the upper half-plane
$\HH$.  The general strategy is to rewrite the integrands as
$\partial_{\bar z}$-derivatives of suitable boundary primitives and
then apply the complex Green's theorem.

To handle the singularities on the real line, we use an indented
domain.  Throughout, let $b_t=\sqrt{4t/(\alpha(1-\alpha))}$ for $t>0$ and $c=2\alpha-1$; the case $t=0$ is immediate.  Recall that
$\phi(z)=(z+(1-\alpha))^\alpha(z-\alpha)^{1-\alpha}$.  \review{The distinguished boundary points are the branch points} at $z=\alpha-1$ and $z=\alpha$, and
the simple zero of $\phi'$ at $z=c$, where the boundary
primitives below have singular terms.  For small
$\eps>0$ and large $R>0$, we define the truncated and indented upper
half-plane
\begin{align*}
\HH_{\eps,R}
:=
\{z\in\HH:
|z|<R,\ |z-(\alpha-1)|>\eps,\ |z-(2\alpha-1)|>\eps,\ |z-\alpha|>\eps\}.
\end{align*}
Its boundary $\partial\HH_{\eps,R}$ consists of the large semicircle
$\Gamma_R$, three small semicircles
$\Gamma_{\alpha-1},\Gamma_{2\alpha-1},\Gamma_\alpha$ (oriented
clockwise), and the disjoint real segments connecting them.

The truncation $|z|<R$ is legitimate although the principal
values were defined by a circular truncation in the physical variable:
under the change of variables $Z=b_t\phi(w)$ used below, the circle
$|Z|=b_tR$ corresponds to the curve $|\phi(w)|=R$, which lies in an
annulus of thickness $O(R^{-1})$ around $|w|=R$ because
$\phi(w)=w+O\bigl(|w|^{-1}\bigr)$.  Both area integrands below are
$O\bigl(|w|^{-2}\bigr)$ near infinity, so the integral over the
symmetric difference of the two truncations is $O(R^{-2})$ and tends
to $0$.  The circular principal values can therefore be computed on
$\HH_{\eps,R}$: the area integrands are locally integrable at the
three boundary points by the local behavior of $\phi$ and $\phi'$,
so the removed half-disks contribute terms tending to zero.

\begin{proof}[Proof of Lemma~\ref{lemma-QC-deform-1}]
By performing the change of variables
$z=g_t^{-1}(b_t w)=b_t\phi(w)$
and defining the normalized integral $I$, we can rewrite
$\mathcal L(\nu)$ as
\begin{align}\label{eqn-L-change-of-variables}
\mathcal L(\nu)
=
-\frac{4a\sqrt t}{\pi\sqrt{\alpha(1-\alpha)}}\,I,
\end{align}
where
\[
I:=
\operatorname{Re}\PV\iint_{\HH}
\frac{\overline{\phi(z)}^{\,n-1}}{\phi(z)^{n-1}}
\left(
\frac{1}{\phi'(z)\bigl(z-(2\alpha-1)\bigr)}
-\frac{\phi'(z)}{\phi(z)}
\right)
\overline{\phi'(z)}\dA (z),
\]
and the prefactor comes from the change of variables
$z=b_t\phi(w)$: the area element
contributes $b_t^2|\phi'|^2$, the bracket contributes $b_t^{-1}$, the
coefficient $\nu$ contributes the factor $a$, and the definition of
$\mathcal L$ contributes $-2/\pi$, giving $-(2/\pi)ab_t$ in total.
Define the function
\[
W(z)
:=
\frac{1}{\phi(z)^{n-1}}
\left(
\frac{1}{\phi'(z)\bigl(z-(2\alpha-1)\bigr)}
-\frac{\phi'(z)}{\phi(z)}
\right),
\]
which is holomorphic on $\HH$.  The integral $I$ is exactly
\[
I=\frac1n\operatorname{Re}\PV\iint_{\HH}
\partial_{\bar z}\bigl(W(z)\overline{\phi(z)}^{\,n}\bigr)\dA (z).
\]
Applying Green's theorem over $\HH_{\eps,R}$ and taking the real part
yields
\begin{align*}
I=
\frac{1}{2n}
\lim_{R\to\infty,\,\eps\to0}
\operatorname{Im}
\int_{\partial\HH_{\eps,R}}
W(z)\overline{\phi(z)}^{\,n}\,dz .
\end{align*}
We evaluate the contour integral piece by piece.

\begin{itemize}
\item[\textbf{(1)}]
\textbf{The real segments.}
Outside the interval $(\alpha-1,\alpha)$, $\phi(z)$ is purely real.
Hence $W(z)\overline{\phi(z)}^{\,n}\in\R$, contributing zero to the
imaginary part.  Over the interval $(\alpha-1,\alpha)$,
$\phi(x)=|\phi(x)|e^{i\pi(1-\alpha)}$.  The integrand splits into two
terms.  Direct integration of the rational term and the derivative
term over the two remaining pieces of $(\alpha-1,\alpha)$ yields
\begin{align}\label{eqn-L-real-pv}
\begin{split}
\Bigl(\int_{\alpha-1+\eps}^{c-\eps}
+\int_{c+\eps}^{\alpha-\eps}\Bigr)
W(x)\overline{\phi(x)}^{n}dx
=
-e^{i2n\pi\alpha}
\Bigl[
\frac{2\alpha(1-\alpha)}{\eps}
-2+(1-2\alpha)\ln\frac{1-\alpha}{\alpha}
\Bigr]
+O(\eps^{1-\alpha}).
\end{split}
\end{align}
The remainder is $O(\eps^{1-\alpha})$ by the local behavior of
$\phi$ at the branch points, since $\alpha\in(1/2,1)$.
This is a truncated integral and not a finite principal value by
itself; the $1/\eps$ term cancels against the central indentation in
step~(4) below.

\item[\textbf{(2)}]
\textbf{The large semicircle $\Gamma_R$.}
As $R\to\infty$ we have $\phi(z)=z+O\bigl(|z|^{-1}\bigr)$ and
$\phi'(z)/\phi(z)=(z-c)/\bigl((z-\alpha+1)(z-\alpha)\bigr)$ with
$c=2\alpha-1$, whence a direct expansion gives
\[
\phi(z)\left(
\frac{1}{\phi'(z)(z-c)}-\frac{\phi'(z)}{\phi(z)}
\right)
=
\frac{c}{z}+O\bigl(|z|^{-2}\bigr),
\]
and
\[
\left(\frac{\overline{\phi(z)}}{\phi(z)}\right)^{n}
=
\left(\frac{\bar z}{z}\right)^{n}\bigl(1+O\bigl(|z|^{-2}\bigr)\bigr),
\]
so that
\[
W(z)\overline{\phi(z)}^{\,n}
=
c\,\frac{\bar z^{\,n}}{z^{n+1}}+O\bigl(|z|^{-2}\bigr).
\]
Writing $z=Re^{i\theta}$ and $dz=iRe^{i\theta}d\theta$,
\[
\int_{\Gamma_R}W(z)\overline{\phi(z)}^{\,n}\,dz
=
ic\int_0^\pi e^{-2in\theta}\,d\theta+O\bigl(R^{-1}\bigr)
=
O\bigl(R^{-1}\bigr)\longrightarrow0 ,
\]
since the oscillatory integral, differential included, vanishes for
every integer $n\ge1$.

\item[\textbf{(3)}]
\textbf{The branch point semicircles $\Gamma_{\alpha-1}$ and
$\Gamma_\alpha$.}
Near $z=\alpha-1$, the local bounds
$|\phi(z)|=O(\eps^\alpha)$ and $\phi'(z)=O(\eps^{\alpha-1})$ imply
$|W(z)\overline{\phi(z)}^{\,n}|=O(\eps^{\alpha-1})$.  The integral is
thus bounded by $O(\eps^{\alpha-1})\cdot\pi\eps=O(\eps^\alpha)\to0$.
By symmetry, the integral over $\Gamma_\alpha$ is
$O(\eps^{1-\alpha})\to0$.

\item[\textbf{(4)}]
\textbf{The central indentation $\Gamma_{2\alpha-1}$.}
Since $\phi'$ has a simple zero at $c$, the boundary primitive
$W(z)\overline{\phi(z)}^{\,n}$ has terms of orders $(z-c)^{-2}$
and $(z-c)^{-1}$, with a bounded remainder.  Integrating these terms
clockwise from $\pi$ to $0$ gives
\begin{align}\label{eqn-L-pole}
\int_{\Gamma_{2\alpha-1}}
W(z)\overline{\phi(z)}^{\,n}\,dz
=
e^{i2n\pi\alpha}
\left[
\frac{2\alpha(1-\alpha)}{\eps}
+\pi i(1-2\alpha)
\right]
+O(\eps).
\end{align}

\item[\textbf{(5)}]
\textbf{Conclusion.}
Summing \eqref{eqn-L-real-pv} and \eqref{eqn-L-pole}, the $1/\eps$
terms cancel.  Taking the limit as $\eps\to0$ and extracting the
imaginary part gives
\begin{align*}
I
&=
\frac{1}{2n}\operatorname{Im}
\left\{
e^{i2n\pi\alpha}
\left[
\pi i(1-2\alpha)
+2-(1-2\alpha)\ln\frac{1-\alpha}{\alpha}
\right]
\right\}
\\
&=
\frac{\pi}{2n}
\left[
(1-2\alpha)\cos(2n\pi\alpha)
+\frac{\sin(2n\pi\alpha)}{\pi}
\Bigl(2-(2\alpha-1)\ln\frac{\alpha}{1-\alpha}\Bigr)
\right].
\end{align*}
Multiplying $I$ by the prefactor
$-\frac{4a\sqrt t}{\pi\sqrt{\alpha(1-\alpha)}}$ from
\eqref{eqn-L-change-of-variables} gives the right-hand side of
\eqref{eqn-QC-deform-4}, completing the proof.\qedhere
\end{itemize}
\end{proof}

\begin{proof}[Proof of Lemma~\ref{lemma-QC-deform-2}]
We need to prove that for
$\nu(z)=a(\bar z/z)^{n-1}$ with $a\in(-1,1)$ and $n\ge1$,
\begin{align}
\mathcal S(\nu)
&:=
\frac1\pi\operatorname{Re}\PV\iint_{\HH}
\nu(z)\bigl(g_t'(z)^2-1\bigr)\,dA(z)
\notag\\
&=
\frac{ta}{n}
\left[
2\cos(2n\pi\alpha)
+\frac{\sin(2n\pi\alpha)}{\pi}
\Bigl(
\frac{2\alpha-1}{\alpha(1-\alpha)}
+2\ln\frac{\alpha}{1-\alpha}
\Bigr)
\right].
\label{eqn-S-target-appendix}
\end{align}

Applying the change of variables
$z=g_t^{-1}(b_t w)=b_t\phi(w)$, the area
measure transforms as
$dA(z)=\frac{4t}{\alpha(1-\alpha)}|\phi'(w)|^2\,dA(w)$ and the
derivative satisfies $g_t'(z)=1/\phi'(w)$.  Substituting these into
$\mathcal S(\nu)$ and rewriting the integrand, we obtain
\begin{align*}
\mathcal S(\nu)
&=
\frac{4ta}{\pi\alpha(1-\alpha)}
\operatorname{Re}\PV\iint_{\HH}
\frac{\overline{\phi(z)}^{\,n-1}}{\phi(z)^{n-1}}
\left(
\frac{1}{\phi'(z)^2}-1
\right)
|\phi'(z)|^2\dA (z)
\\
&=
\frac{4ta}{\pi\alpha(1-\alpha)}
\underbrace{
\operatorname{Re}\PV\iint_{\HH}
\frac{\overline{\phi(z)}^{\,n-1}}{\phi(z)^{n-1}}
\left(
\frac{\overline{\phi'(z)}}{\phi'(z)}
-\phi'(z)\overline{\phi'(z)}
\right)\dA (z)
}_{:=\,J}.
\end{align*}

Taking the logarithmic derivative of
$\phi(z)=(z+(1-\alpha))^\alpha(z-\alpha)^{1-\alpha}$ yields
\begin{align*}
\frac{\phi'(z)}{\phi(z)}
=
\frac{z-(2\alpha-1)}{(z-\alpha+1)(z-\alpha)} .
\end{align*}
Define
\[
V(z):=
\frac{1}{\phi(z)^{n-1}}
\left(
\frac{1}{\phi'(z)}-\phi'(z)
\right)
\]
(the factor $1/n$ is not part of $V$; it appears only in
Green's formula below).
Direct differentiation gives
\[
\partial_{\bar z}
\left(
V(z)\overline{\phi(z)}^{\,n}
\right)
=
n
\frac{\overline{\phi(z)}^{\,n-1}}{\phi(z)^{n-1}}
\left(
\frac{\overline{\phi'(z)}}{\phi'(z)}
-\phi'(z)\overline{\phi'(z)}
\right).
\]
Hence
$J=\frac1n\operatorname{Re}\PV\iint_{\HH}
\partial_{\bar z}\bigl(V(z)\overline{\phi(z)}^{\,n}\bigr)\dA (z)$.
Applying Green's theorem on the indented domain $\HH_{\eps,R}$ as in
Lemma~\ref{lemma-QC-deform-1}, we obtain
\begin{align*}
J=
\frac{1}{2n}
\lim_{R\to\infty,\,\eps\to0}
\operatorname{Im}
\int_{\partial\HH_{\eps,R}}
V(z)\overline{\phi(z)}^{\,n}\,dz .
\end{align*}

To explicitly evaluate the boundary integrals, we decompose
$V(z)\overline{\phi(z)}^{\,n}$ using the rational functions $Q$ and $P$:
\begin{align}\label{eqn-V-decomp}
V(z)\overline{\phi(z)}^{\,n}
=
\left(
\frac{\overline{\phi(z)}}{\phi(z)}
\right)^{n}Q(z)
-
\frac{\overline{\phi(z)}^{\,n}}{\phi(z)^{n-2}}\,P(z),
\end{align}
where
\begin{align*}
Q(z):=\frac{(z-\alpha+1)(z-\alpha)}{z-(2\alpha-1)},
\qquad
P(z):=\frac{1}{Q(z)}=\frac{z-(2\alpha-1)}{(z-\alpha+1)(z-\alpha)} .
\end{align*}
We now analyze the contour integrals over each part of
$\partial\HH_{\eps,R}$.

\begin{itemize}
\item[\textbf{(1)}]
\textbf{The large semicircle $\Gamma_R$.}
As $z\to\infty$, the expansion
$\phi(z)=z-\frac{\alpha(1-\alpha)}{2z}+O(1/z^2)$ implies
$\phi'(z)=1+\frac{\alpha(1-\alpha)}{2z^2}+O(1/z^3)$.  Consequently,
the constant terms in $1/\phi'(z)-\phi'(z)$ cancel:
\[
\frac{1}{\phi'(z)}-\phi'(z)
=
-\frac{\alpha(1-\alpha)}{z^2}+O(1/z^3).
\]
Thus, the boundary integrand has the asymptotic expansion
\[
V(z)\overline{\phi(z)}^{\,n}
=
-\alpha(1-\alpha)\,
\frac{\bar z^{\,n}}{z^{n+1}}
+O(|z|^{-2}).
\]
Integrating over $\Gamma_R$ (where $z=Re^{i\theta}$), the
$O(|z|^{-2})$ error term contributes $O(1/R)$.  The leading term
evaluates to
\[
-\alpha(1-\alpha)
\int_0^\pi e^{-2in\theta}\,i\,d\theta=0,
\]
since $n\ge1$ is an integer.  Hence
$\lim_{R\to\infty}
\int_{\Gamma_R}V(z)\overline{\phi(z)}^{\,n}\,dz=0$.

\item[\textbf{(2)}]
\textbf{The branch point semicircles $\Gamma_{\alpha-1}$ and
$\Gamma_\alpha$.}
Parametrize $\Gamma_{\alpha-1}$ by $z=\alpha-1+\eps e^{i\theta}$,
giving $|dz|=\eps\,d\theta$.  Locally,
$|\phi(z)|=O(\eps^\alpha)$, $|Q(z)|=O(\eps)$, and
$|P(z)|=O(\eps^{-1})$.  Bounding the absolute value of
\eqref{eqn-V-decomp}, the integral of the first term is bounded by
$\int O(\eps)\,\eps\,d\theta=O(\eps^2)$, and the second term by
$\int O(\eps^{2\alpha})O(\eps^{-1})\,\eps\,d\theta
=O(\eps^{2\alpha})$.  Since $\alpha\in(0,1)$, this implies that the
integral over $\Gamma_{\alpha-1}$ vanishes as $\eps\to0$.  A
symmetric local analysis near $z=\alpha$ shows that the integral is
bounded by $O(\eps^2)+O(\eps^{2-2\alpha})$, which tends to zero since
$2-2\alpha>0$.

\item[\textbf{(3)}]
\textbf{The central small semicircle $\Gamma_{2\alpha-1}$.}
Parametrize $\Gamma_{2\alpha-1}$ by
$z=2\alpha-1+\eps e^{i\theta}$ for $\theta\in[\pi,0]$.  Here
$P(2\alpha-1)=0$, making the second term $O(\eps)$ whose integral is
$O(\eps^2)\to0$.  For the first term, $\phi'(2\alpha-1)=0$, so the
linear variation vanishes and
\[
\left(
\frac{\overline{\phi(z)}}{\phi(z)}
\right)^{n}
=
e^{i2n\pi\alpha}+O(\eps^2).
\]
Combining this with
$Q(z)=-\frac{\alpha(1-\alpha)}{\eps e^{i\theta}}+O(1)$, we obtain
\[
V(z)\overline{\phi(z)}^{\,n}
=
e^{i2n\pi\alpha}\,
\frac{-\alpha(1-\alpha)}{\eps e^{i\theta}}
+O(1).
\]
Integrating clockwise from $\pi$ to $0$:
\begin{align}\label{eqn-gamma-mid-capacity}
\begin{split}
\int_{\Gamma_{2\alpha-1}}
V(z)\overline{\phi(z)}^{\,n}\,dz
&=
\int_\pi^0
e^{i2n\pi\alpha}\,
\frac{-\alpha(1-\alpha)}{\eps e^{i\theta}}
\bigl(i\eps e^{i\theta}\bigr)\,d\theta
+O(\eps)\\
&=
i\pi\alpha(1-\alpha)e^{i2n\pi\alpha}+O(\eps).
\end{split}
\end{align}

\item[\textbf{(4)}]
\textbf{The principal value on the real line.}
Outside $(\alpha-1,\alpha)$, $\phi(x)$ is real, so
$V(x)\overline{\phi(x)}^{\,n}\in\R$, contributing zero to the
imaginary part.  Inside $(\alpha-1,\alpha)$, we have
$\bigl(\overline{\phi(x)}/\phi(x)\bigr)^n=e^{i2n\pi\alpha}$ and
$\overline{\phi(x)}^{\,n}/\phi(x)^{n-2}
=e^{i2(n-1)\pi\alpha}|\phi(x)|^2$.  Thus
\begin{align}\label{eqn-pv-real}
\begin{split}
\PV\int_{\alpha-1}^{\alpha}
V(x)\overline{\phi(x)}^{\,n}\,dx
&=
e^{i2n\pi\alpha}\PV\int_{\alpha-1}^{\alpha}Q(x)\,dx\\
&\quad
-
e^{i2(n-1)\pi\alpha}
\int_{\alpha-1}^{\alpha}|\phi(x)|^2P(x)\,dx .
\end{split}
\end{align}
For
$Q(x)=(x-2\alpha+1)+(2\alpha-1)+\frac{\alpha(\alpha-1)}{x-2\alpha+1}$,
the symmetric cancellation of the principal value yields
\begin{align*}
\PV\int_{\alpha-1}^{\alpha}Q(x)\,dx
&=
\left[
\frac12(x-2\alpha+1)^2+(2\alpha-1)x
\right]_{\alpha-1}^{\alpha}
+
\alpha(\alpha-1)\ln\left|\frac{1-\alpha}{-\alpha}\right|
\\
&=
\alpha-\frac12
-\alpha(1-\alpha)\ln\frac{1-\alpha}{\alpha}.
\end{align*}
For the second integral,
$|\phi(x)|^2P(x)
=-(x-\alpha+1)^{2\alpha-1}(\alpha-x)^{1-2\alpha}(x-2\alpha+1)$.
Letting $u=x-\alpha+1$, the Beta integral gives
\begin{align*}
\int_0^1-u^{2\alpha-1}(1-u)^{1-2\alpha}(u-\alpha)\,du
=
-B(2\alpha+1,2-2\alpha)+\alpha B(2\alpha,2-2\alpha)
=
0,
\end{align*}
because
$B(2\alpha+1,2-2\alpha)=\alpha\Gamma(2\alpha)\Gamma(2-2\alpha)
=\alpha B(2\alpha,2-2\alpha)$.
Thus this integral vanishes.

\item[\textbf{(5)}]
\textbf{Conclusion.}
The total imaginary part is generated solely by
\eqref{eqn-gamma-mid-capacity} and \eqref{eqn-pv-real}.  Thus
\begin{align*}
J
&=
\frac{1}{2n}\operatorname{Im}
\left[
i\pi\alpha(1-\alpha)e^{i2n\pi\alpha}
+
e^{i2n\pi\alpha}
\left(
\alpha-\frac12
+\alpha(1-\alpha)\ln\frac{\alpha}{1-\alpha}
\right)
\right]
\\
&=
\frac{1}{2n}
\left[
\pi\alpha(1-\alpha)\cos(2n\pi\alpha)
+
\sin(2n\pi\alpha)
\left(
\alpha-\frac12
+\alpha(1-\alpha)\ln\frac{\alpha}{1-\alpha}
\right)
\right],
\end{align*}
where we used
$-\ln\frac{1-\alpha}{\alpha}=\ln\frac{\alpha}{1-\alpha}$.
Finally, multiplying by the scale factor
$\frac{4ta}{\pi\alpha(1-\alpha)}$, we obtain
\begin{align*}
\mathcal S(\nu)
=
\frac{4ta}{\pi\alpha(1-\alpha)}\,J
&=
\frac{2ta}{\pi n}
\Bigl[
\pi\cos(2n\pi\alpha)
+
\sin(2n\pi\alpha)
\Bigl(
\frac{2\alpha-1}{2\alpha(1-\alpha)}
+\ln\frac{\alpha}{1-\alpha}
\Bigr)
\Bigr]
\\
&=
\frac{ta}{n}
\Bigl[
2\cos(2n\pi\alpha)+
\frac{\sin(2n\pi\alpha)}{\pi}
\Bigl(
\frac{2\alpha-1}{\alpha(1-\alpha)}
+2\ln\frac{\alpha}{1-\alpha}
\Bigr)
\Bigr].
\end{align*}
This exactly matches the desired formula
\eqref{eqn-S-target-appendix}, and the proof is complete.\qedhere
\end{itemize}
\end{proof}
\section{Winding numbers}\label{appendix B}

Let $\gamma:[0,1]\to\mathbb{C}$ be a closed curve and let
$P\in\mathbb{C}\setminus\gamma([0,1])$. We write
$\operatorname{Ind}(\gamma,P)$ for the winding number of
$\gamma$ about $P$. We will use the following standard properties of $\operatorname{Ind}(\gamma,P)$.

\begin{lemma}[{\cite[Corollary~3.8]{Fu95}}]\label{lem-winding-homotopy}
Let $P\in\C$, and let \(\alpha,\beta\colon[0,1]\to\C\setminus\{P\}\)
be closed curves.  If $\alpha$ and $\beta$ are homotopic as closed
curves in $\C\setminus\{P\}$, then
\[
\Ind(\alpha,P)=\Ind(\beta,P).
\]
\end{lemma}

\begin{lemma}[{\cite[Lemma~3.2.6]{Ro15}}]
\label{lem-winding-product}
Let $\alpha,\beta\colon[0,1]\to\C\setminus\{0\}$ be continuous closed
curves.  Define their pointwise product
$\alpha\beta\colon[0,1]\to\C\setminus\{0\}$ by
$(\alpha\beta)(t):=\alpha(t)\beta(t)$, $t\in[0,1]$.  Then
\[
\Ind(\alpha\beta,0)
=
\Ind(\alpha,0)+\Ind(\beta,0).
\]
\end{lemma}

\section*{Acknowledgments}
\addcontentsline{toc}{section}{Acknowledgments}
We thank Professor Guizhen Cui for suggesting that we consider the standard normalization in order to remove the support and decay assumptions on the Beltrami differential. We used AI tools to assist with simplifying arguments and developing
parts of the proofs.

\end{document}